\documentclass[12pt]{amsart} 
\usepackage{amsfonts, amsmath,  amssymb, amsthm,  array, bm, multirow,
hyperref,  pdfsync, graphicx}

\newtheorem{thm}{Theorem}[section] 

\newtheorem{lemma}[thm]{Lemma}  
\newtheorem{cor}[thm]{Corollary}

\theoremstyle{definition}  
\newtheorem{example}[thm]{Example}
\newtheorem{defn}[thm]{Definition}

\theoremstyle{remark}  
\newtheorem{remark}[thm]{Remark}

\newcommand{\calB}{\ensuremath{\mathcal{B}} }
\newcommand{\calC}{\ensuremath{\mathcal{C}} }
\newcommand{\calD}{\ensuremath{\mathcal{D}} }

\newcommand{\calL}{\ensuremath{\mathcal{L}} }

\newcommand{\calN}{\ensuremath{\mathcal{N}} }

\newcommand{\calR}{\ensuremath{\mathcal{R}} }
\newcommand{\calS}{\ensuremath{\mathcal{S}} }

\newcommand{\bfzero}{\ensuremath{\bm{0}} }

\newcommand{\bfa}{\ensuremath{\bm a} }
\newcommand{\bfb}{\ensuremath{\bm b} }

\newcommand{\bfd}{\ensuremath{\bm d} }
\newcommand{\bfe}{\ensuremath{\bm e} }

\newcommand{\bfg}{\ensuremath{\bm g} }

\newcommand{\bfq}{\ensuremath{\bm q} }

\newcommand{\bfs}{\ensuremath{\bm s} }
\newcommand{\bft}{\ensuremath{\bm t} }

\newcommand{\bfv}{\ensuremath{\bm v} }
\newcommand{\bfw}{\ensuremath{\bm w} }
\newcommand{\bfx}{\ensuremath{\bm x} }
\newcommand{\bfy}{\ensuremath{\bm y} }
\newcommand{\bfz}{\ensuremath{\bm z} }

\newcommand{\frakg}{\ensuremath{\mathfrak{g}} }

\newcommand{\boldN}{\ensuremath{\mathbb N}}
\newcommand{\boldP}{\ensuremath{\mathbb P}}

\newcommand{\boldR}{\ensuremath{\mathbb R}}
\newcommand{\boldZ}{\ensuremath{\mathbb Z}}

\newcommand{\ad}{\operatorname{ad}}

\newcommand{\Aut}{\operatorname{Aut}}

\newcommand{\col}{\operatorname{Col}}

 \newcommand{\diag}{\operatorname{diag}}

 \newcommand{\Exp}{\operatorname{Exp}}

 \newcommand{\myspan}{\operatorname{span}}
 \newcommand{\Null}{\operatorname{Null}}

 \newcommand{\sgn}{\operatorname{sgn}}
 \newcommand{\sign}{\operatorname{sign}}

 \newcommand{\Ln}{\operatorname{Ln}}
 \newcommand{\E}{\operatorname{E}}

\newcommand{\la}{\langle} 
\newcommand{\ra}{\rangle}

\begin{document}

 \title[Parametrizing varieties of Lie algebras]{Methods for
   Parametrizing varieties of Lie algebras}

 \email{payntrac@isu.edu}

\subjclass{Primary: 17B30; Secondary:   22-04, 22E15, 22E25}

\begin{abstract}    
  A real $n$-dimensional anticommutative nonassociative algebra is
  represented by an element of $\wedge^2(\boldR^n)^\ast \otimes
  \boldR^n.$ For each $\mu \in \wedge^2(\boldR^n)^\ast \otimes
  \boldR^n,$ there is a unique subset $\Lambda \subseteq \{ (i,j,k) \,
  : \,1 \le i < j \le n , 1 \le k \le n\}$ so that the structure
  constant $\mu_{ij}^k$ with $i < j$ is nonzero if and only if
  $(i,j,k) \in \Lambda.$ The set of all $\mu \in
  \wedge^2(\boldR^n)^\ast \otimes \boldR^n$ with subset $\Lambda$ is
  denoted $\calS_\Lambda(\boldR);$ these sets stratify
  $\wedge^2(\boldR^n)^\ast \otimes \boldR^n.$ We describe how to
  smoothly parametrize the Lie algebras in $\calS_\Lambda(\boldR)$ up
  to isomorphism, for $\Lambda$ satisfying certain frequently seen
  hypotheses.   
\end{abstract} 

\maketitle

 \medskip

 \noindent{\em Keywords:  varieties of  Lie algebras,  nilpotent Lie algebra}

 \medskip

 \noindent{\em Corresponding author: Tracy L.\ Payne, Department of Mathematics,  Idaho State University, 
 921 S. 8th Ave.,
 Pocatello, ID 83209-8085, USA.} {\tt payntrac@isu.edu} {\em (208) 282-3650, FAX (208) 282-2636 }

 \medskip


 \section{Introduction}\label{introduction}


 In many areas of mathematics, it is natural to subdivide an object
into smaller, simpler pieces and to parametrize each piece in some
useful, controlled way (\cite {yomdin-14}).  A {\em semi-algebraic
set} is a subset of $\boldR^n$ defined by a finite number of
polynomial equations and inequalities, along with the operations of
intersection and union.  Every closed and bounded semi-algebraic set
is semi-algebraically triangulable (\cite{benedetti-risler}).  A {\em
parametrization} of semi-algebraic subset $A$ of $\boldR^n$ is a
collection of semi-algebraic subsets $A_j$ that cover $A,$ along with
surjective {\em charts} $\varphi_j : I^{n_j} \to A_j,$ where $I^{n_j}$
is a cube in $\boldR^{n_j},$ such that each chart $\varphi_j$ is
algebraic and is a homeomorphism from the interior of $I^{n_j}$ to the
interior of $A_j.$

 Here we are interested in parametrizing varieties of Lie algebras.
These sets are not semi-algebraic; rather, they are quotients of
semi-algebraic sets by the action of a Lie group.  Let $A$ be a
semi-algebraic subset of $\boldR^n$ and let $A/{\sim}$ be the quotient
of $A$ under the action of a group.  Let $\{A_j\}$ be a collection of
invariant semi-algebraic subsets that cover $A.$ A {\em
parametrization} of $A/{\sim}$ is the collection $\{A_j/{\sim}\}$ of
semi-algebraic quotients, along with {\em charts} $\varphi_j : I^{n_j}
\to A_j/{\sim},$ such that each chart $\varphi_j$ is algebraic and is
a homeomorphism from the interior of $I^{n_j}$ to the interior of
$A_j/{\sim}.$ Our goal is to find explicit parametrizations of
specific subsets of the variety $\widetilde \calN_n(\boldR)$ of real
nilpotent Lie algebras of dimension $n.$

 Real nilpotent Lie algebras of dimension 7 and lower have been
classified (\cite{morozov-58, seeley-93, gong-98}).  For $n \le 6,$
$\widetilde \calN_n(\boldR)$ is discrete.  The space $\widetilde
\calN_7(\boldR)$ is the union of isolated points and one-parameter
families.  In dimension eight and higher, nilpotent Lie algebras have
not been classified, and there are many components of $\widetilde
\calN_n(\boldR)$ of higher dimension (\cite{nil-dim8-moduli}).
Furthermore, there are large families of ``characteristically
nilpotent'' Lie algebras (\cite{hakimjanov}); these do not admit
nontrivial semisimple derivations.  Due to these complications, in
dimension eight and higher, instead of analyzing all of $\widetilde
\calN_n(\boldR),$ it is natural to focus on a tractable subset $\calR$
of $\widetilde \calN_n(\boldR).$ For example, $\calR$ might be
\begin{itemize}
\item  the set of two-step nilpotent Lie algebras  
\item  the set of filiform or quasi-filiform Lie algebras 
\item  the set of  $\boldN$-graded or naturally graded nilpotent Lie
  algebras 
\item a set of nilpotent Lie algebras that admit a special kind of
  structure (affine, symplectic, contact, almost-complex, K\"ahler, etc.).
\end{itemize}  
Or, $\calR$ might be defined as the intersection of two or more such
sets.  

The subclass of filiform Lie algebras is not discrete in dimensions
seven and higher and has been the setting for many kinds of
classification problems. (See, for example, \cite{vergnethesis},
\cite{millionschikov}, \cite{burde-06}, \cite{arroyo-11}.)  The
classification of complex two-step nilpotent Lie algebras of dimension
9 and lower was completed in \cite{galitski-timashev}, and dimensions
of moduli spaces of two-step nilpotent algebras of types $(p,q)$ were
found in \cite{eberlein-03}.

  Another subclass of nilpotent Lie algebras is the set of nilpotent
Lie algebras whose ``Nikolayevsky derivation'' has positive
eigenvalues all of multiplicity one.  (See \cite{nikolayevsky-simple}
for a definition of the Nikolayevsky derivation.)  A subset of this
subclass in dimensions 7 and 8 was classified in
\cite{kadioglu-payne}.  In fact, the motivation for this work is to
develop the necessary tools for the completion in \cite{payne-comp-II}
of the classification begun in \cite{kadioglu-payne}.  In
\cite{payne-comp-II}, we classify all nilpotent Lie algebras of
dimensions 7 and 8 for which the Nikolayevsky derivation is simple
with positive eigenvalues.  The methods developed here are essential
in \cite{payne-comp-II} for the determination of isomorphism classes
of nilpotent Lie algebras and for the parametrization of continuous
families of those isomorphism classes.

Our goal is to find parametrizations of subclasses of the space
$\widetilde \calN_n(\boldR)$ of real $n$-dimensional nilpotent Lie
algebras, or, more generally, the space $\widetilde \calL_n(K)$ of
$n$-dimensional Lie algebras over the field $K.$ In attempting to
classify some family of Lie algebras up to isomorphism, one frequently
shows that there is some efficient basis with respect to which nonzero
structure constants for algebras in the class must have specified
indices, and then one determines which of the algebras with such
structure constants are isomorphic.

Oftentimes a subclass itself may be completely described in terms of
indices of nonvanishing structure constants.  To be precise, given a
Lie algebra with basis $\calB,$ there exists a subset $\Lambda$ of
\begin{equation}\label{upsilon n} \Upsilon_n = \{ (i,j,k) \, : \,
i,j,k \in [n], i< j\} \subseteq [n]^3,
\end{equation} where we denote the set $\{1, 2, \ldots, n\}$ by $[n],$
so that the structure constant $\alpha_{ij}^k$ (with $i<j$) is nonzero
if and only if $(i,j,k) \in \Lambda.$ The subclass $\widetilde
\calL_\Lambda(K)$ of $\widetilde \calL_n(K)$ is defined to be the set
of all Lie algebras over $K$ that admit a basis $\calB$ so the
structure constants with respect to $\calB$ are indexed by $\Lambda.$
Such sets are quotients of semi-algebraic subsets of
$\wedge^2(K^n)^\ast \otimes K^n.$ Often, subsets $\calR$ of
$\widetilde \calL_n(K)$ may be written as the union of sets
$\widetilde \calL_\Lambda(K)$ as $\Lambda$ varies over a subset of
$\Upsilon_n.$

For example, let $\calR = \widetilde \calN_n(K),$ the set of
$n$-dimensional nilpotent Lie algebras over $K.$ If the Lie algebra
$\frakg_\alpha$ is nilpotent, then by Engel's Theorem there is a basis
$\calB = \{ x_i\}_{i=1}^n$ for $\frakg_\alpha$ with respect to which
the operators $\ad_{x_i}$ are simultaneously upper triangularizable.
We say that such a basis $\calB$ is a {\em triangular basis}. Relative
to $\calB,$ the structure constants for $\frakg_\alpha$ (modulo
skew-symmetry) are in the set
\[ \Theta_n = \{(i,j,k) \in [n]^3\, : \, i < j < k \} \subseteq
\Upsilon_n. \] Hence every $n$-dimensional nonabelian nilpotent Lie
algebra over $K$ is in the set $\cup_{\Lambda \subseteq \Theta_n}
\widetilde \calL_\Lambda(K).$ Conversely, the nilpotency condition
holds automatically for every Lie algebra $\frakg_\alpha$ with
structure constants indexed by a subset of $\Theta_n$ relative to some
basis.  Many other classes of nilpotent Lie algebras, including
filiform and $\boldN$-graded Lie algebras, may be expressed as the
union of sets $\widetilde \calL_\Lambda(K).$

To parametrize such a set $\calR = \cup_{\Lambda \in S} \widetilde
\calL_\Lambda(K),$ we will find a parametrization of each set
$\widetilde \calL_\Lambda(K).$ We assume that the sets $\widetilde
\calL_\Lambda(K)$ are disjoint and neglect the issue of when
isomorphic Lie algebras occur in different sets $\widetilde
\calL_\Lambda(K);$ in practice, the set $S$ may often be chosen so
that each Lie algebra in $\calR$ occurs in exactly one $\widetilde
\calL_\Lambda(K)$ with $\Lambda \in S.$

 Although we are primarily motivated by the goal of understanding the
class of real nilpotent Lie algebras, we will work in a more general
setting.  We will drop the constraint imposed by the Jacobi Identity
and consider the set $\calS_\Lambda(K)$ of all anticommutative
nonassociative algebras over a field $K$ whose structure constants are
indexed by an index set $\Lambda.$ We will first consider the question
of when products having the same index set $\Lambda$ are isomorphic.
We will then consider which products having the index set $\Lambda$
satisfy Jacobi Identity.  This will enable us to find a compact
semi-algebraic subset $\Sigma$ of $\boldR^{|\Lambda|}$ whose interior
is homeomorphic to $\widetilde \calL_\Lambda(\boldR).$ Then a
parametrization of the set $\Sigma$ yields the desired parametrization
for $\widetilde \calL_\Lambda(\boldR).$

The parametrization depends largely on the combinatorics of the index
set $\Lambda.$ It turns out that many index sets $\Lambda \subseteq
\Theta_n$ have the same combinatorial relationships, so the
descriptions of the parametrizations of the sets $\widetilde
\calL_\Lambda(\boldR),$ with $\Lambda \subseteq \Theta_n,$ fall into a
small number of classes in each dimension.

It should also be mentioned that the methods presented here provide a
way to test whether two Lie algebras having the same index set are
isomorphic.

The paper is organized as follows.  In the next section, we state the
main results.  In Section \ref{preliminaries}, necessary background
material is covered, and our method of parametrizing classes of Lie
algebras is outlined.  In Section \ref{cross sections} we focus on
defining a subset $\Sigma$ of $\calS_\Lambda(\boldR)$ so that each Lie
algebra in $\calS_\Lambda(\boldR)$ is isomorphic to exactly one Lie
algebra in $\Sigma.$ In Section \ref{combinatorics}, we define some
combinatorial objects associated to a set $\Lambda$ and describe some
useful properties they have.  In Section \ref{jacobi identity}, we
address issues related to the Jacobi Identity and find how to
parametrize sets of form $\widetilde \calL_\Lambda(\boldR).$ Due to
the many technicalities involved, and our orientation to applications,
we include many concrete examples along the way.

\section{Main results}

\subsection{Necessary definitions}
Let $K$ be a field.  For $\alpha \in \wedge^2(K^n)^\ast \otimes K^n,$
we let $\frakg_\alpha$ denote the anticommutative nonassociative
algebra whose product is defined by $\alpha.$ Let $\calB = \{
x_i\}_{i=1}^n$ be a basis for $K^n,$ and for $\alpha \in
\wedge^2(K^n)^\ast \otimes K^n,$ let $\alpha_{ij}^k$ denote the
structure constants for $\alpha$ relative to $\calB.$

For a subset $\Lambda$  of $\Upsilon_n,$ define the subset 
$\calS_\Lambda(K)$ of $\wedge^2(K^n)^\ast \otimes K^n$ by
\begin{align*}
 \calS_\Lambda(K) = \{ \alpha \in \wedge^2(K^n)^\ast \otimes K^n \, : \,
 i < j \, \text{ and}  \enskip \alpha_{ij}^k \ne 0 & 
\enskip \text{if and only if} \\ 
(i,j,k) \in \Lambda \}. \end{align*}
Allowing $\Lambda$ to vary over all subsets of $\Upsilon_n,$ we obtain
a  semi-algebraic stratification
\begin{equation}\label{partition}
\wedge^2(K^n)^\ast \otimes K^n = \bigcup_{\Lambda \subseteq
  \Upsilon_n} \calS_\Lambda(K)
\end{equation}
of $\wedge^2(K^n)^\ast \otimes K^n.$   

Let $[n]= \{1, 2, \ldots, n\}.$ After taking the dictionary ordering
on $[n]^3,$ the set
\[ \calC = \{ (x_i^\ast \wedge x_j^\ast) \otimes x_k : (i,j,k) \in
\Upsilon_n\} \] 
is an ordered basis for $\wedge^2(K^n)^\ast \otimes K^n.$ Given a
vector $\alpha$ in $\calS_\Lambda(K) \subseteq \wedge^2(K^n)^\ast
\otimes K^n,$ the nonzero entries of its coordinate vector with
respect to $\calC$ are indexed by $\Lambda;$ for this reason we call
the sets $\Lambda$ {\em index sets}.  The coordinate vectors define a
bijection between $\calS_\Lambda(K)$ and $(K \setminus
\{0\})^{|\Lambda|}.$ We sometimes implicitly identify a product in
$\calS_\Lambda(K)$ with a vector in $(K \setminus \{0\})^{|\Lambda|}.$

Let $\calL_n(K)$ be the subset of $\wedge^2(K^n)^\ast \otimes K^n$
consisting of those bilinear maps in $\wedge^2(K^n)^\ast \otimes K^n$
that define products that satisfy the Jacobi Identity.  The
stratification in \eqref{partition} restricts to a stratification of
$\calL_n(K):$ If we let
\[ \calL_\Lambda(K) = \calS_\Lambda(K) \cap \calL_n(K) \] be the set
of Lie brackets whose structure constants are indexed by $\Lambda,$
then
\begin{equation}\label{decomposition} 
\calL_n(K) = \bigcup_{\Lambda \subseteq \Upsilon_n} \calL_\Lambda(K) 
\end{equation} is a
stratification of the set $\calL_n(K)$ of all Lie brackets.
Although each stratum $\calS_\Lambda(K) $ is
nonempty, it is possible that a set $\calL_\Lambda(K) $ is empty.

A nonabelian Lie algebra is represented by many different maps
$\alpha$ in $\calL_n(K).$ The general linear group $GL_n(K)$ acts on
$\wedge^2(K^n)^\ast \otimes K^n,$ with
\begin{equation}\label{GL_n(R) action} (g \cdot \alpha) (x,y) = 
g ( \alpha (g^{-1}x,g^{-1}y) )\end{equation}
for $g \in GL_n(K),$ $\alpha \in \wedge^2(K^n)^\ast \otimes K^n$ and
$x,y \in K^n.$ The space $\widetilde \calL_n(K)$ of all
$n$-dimensional Lie algebras over $K$ is the quotient of $\calL_n(K)$
under this action.  Any choice of initial basis for $K$ will yield the
same quotient space $\widetilde \calL_n(K).$ Therefore, we will often
omit mention of a particular basis when referring to the sets
$\calS_\Lambda(K),$ $\calL_\Lambda(K),$ etc.

\subsection{Statements of main results}

We would like to understand how the combinatorics of a set $\Lambda
\subseteq \Theta_n$ relate to properties of the set of nilpotent Lie
algebras whose structure constants are indexed by $\Lambda.$ For this
purpose, we make some combinatorial definitions in Section
\ref{combinatorics}.  We define what it means for a pair of distinct
triples $\bft_1$ and $\bft_2,$ where $\bft_1, \bft_2 \in \Theta_n
\subseteq [n]^3,$ to be {\em aligned.}  For an aligned pair of triples
$\bft_1$ and $\bft_2$ in $\Theta_n \subseteq [n]^3,$ we define a {\em
quadruple} $q(\bft_1,\bft_2) \in [n]^4,$ and a sign $\sign(\bft_1,
\bft_2) \in \boldZ_2$ for that pair of triples.  To each $\Lambda
\subseteq \Theta_n,$ we associate the set $Q \subseteq [n]^4$ of all
quadruples arising from all aligned pairs of triples $\bft_1, \bft_2
\in \Lambda.$

We would first like to know when $\calL_\Lambda(K)$ is guaranteed to
be empty: what combinatorial conditions on the set $\Lambda$ imply
that no elements of $\calS_\Lambda(K)$ satisfy the Jacobi Identity?
To this end, we formulate the Jacobi Identity for elements of
$\calS_\Lambda(K)$ in the language of triples and quadruples.
\begin{thm}\label{ji-quads-signs} 
Let $\Lambda \subseteq \Theta_n$ be an index set and let 
\[ \bft_1 = (i_1,j_1,k_1), \bft_2 = (i_2,j_2,k_2), \ldots, \bft_m =
(i_m,j_m, k_m)\] be an enumeration of the triples in $\Lambda.$ Let
$Q$ be the set of quadruples for $\Lambda.$ The Jacobi Identity for
elements of $\calS_\Lambda(K)$ is equivalent to the system of
equations
\begin{equation}\label{ji-qt}  
\sum_{q(\bft_p,\bft_r) = (i,j,k,l)}   \sign(\bft_p,\bft_r) \alpha_{i_p j_p}^{k_p}
\alpha_{i_r j_r}^{k_r}= 0, \quad (i,j,k,l) \in Q. \end{equation}
\end{thm}

A corollary to the theorem 
gives a sufficient condition for  $\calL_\Lambda(K)$ to be empty.
\begin{cor}\label{obstruction}
Let $\Lambda \subseteq [n]^3$ be an index set.  
If $\Lambda$ has a  quadruple  of multiplicity one, then no elements
of $\calS_\Lambda(K)$ satisfy the Jacobi Identity.  
\end{cor}

As another corollary to the previous theorem, if the set of quadruples
for an index set $\Lambda$ is empty, then the Jacobi Identity holds
automatically for all elements of $\calS_\Lambda(K).$
\begin{cor}\label{JI automatic}
Let $\Lambda \subseteq [n]^3$ be an index set.  
If $\Lambda$ has no quadruples associated to it, then all  elements
of $\calS_\Lambda(K)$ satisfy the Jacobi Identity.  
\end{cor}

The first corollary is needed in the computational procedure in
\cite{payne-comp-II}. There we systematically analyze sets
$\calS_\Lambda(K),$ checking each one for Lie algebras having certain
properties.  The corollary allows us to disregard all
$\calS_\Lambda(K)$ for which $\Lambda$ has a quadruple of multiplicity
one, thus eliminating irrelevant cases and shortening the computation
time.

Given a stratum $\calS_\Lambda(K)$ of the stratification
\eqref{decomposition} we would like to find a subset $\Sigma$ of
$\calS_\Lambda(K)$ so that each Lie algebra defined by an element of
$\calS_\Lambda(K)$ is isomorphic to precisely one Lie algebra defined
by an element of $\Sigma.$ Here, we work over $\boldR,$ and we assume
that isomorphism classes of Lie algebras in $\calS_\Lambda(\boldR)$
are orbits of the diagonal subgroup $D$ of $GL_n(\boldR)$ under the
action in Equation \eqref{GL_n(R) action}.  We say that a set which
intersects each orbit of an action exactly once is a {\em simple cross
section}.  We seek a simple cross section for the $D$ action so that
Lie algebras in that set will parametrize the isomorphism classes of
Lie algebras in $\calS_\Lambda(\boldR).$

The cross section should be semi-algebraic with compact closure.
Other than that, we'd like some flexibility in our definition of cross
section.  For certain applications we prefer a cross section that
contains a prespecified point $\alpha$ as a ``center point'' at which
all parameter values are zero.  For example, we might be looking at
deformations of a particular Lie algebra of interest.  Or, one might
prefer to choose the center point so that its structure constants are
with respect to a preferred basis.  A third reason to vary the center
point is that the equations encoding the Jacobi Identity may be
simpler in some parametrizations than others.

In Definition 
\ref{Sigma defn} we define a family of bounded subsets $\Sigma(T,S)$ of
$\calS_\Lambda(\boldR)$ depending on a subset $T$ of 
$\boldZ_2^{|\Lambda|}$ and
a subset $S$ of $(\boldR_{>0})^{|\Lambda|}.$
  In Theorem \ref{Fc thm}, we find conditions on a subset 
$\Sigma(T,S)$ of $\calS_\Lambda(\boldR)$ that guarantee that
 it parametrizes algebras in  $\calS_\Lambda(\boldR).$  

\begin{thm}\label{Fc thm}
  Let $\Lambda \subseteq [n]^3$ be an index set of cardinality $m >
  0,$
and let $\dim \Null(Y^T) = d > 0.$
  Assume that any pair of isomorphic Lie algebras in
  $\calS_\Lambda(\boldR)$ lie in the same orbit of the 
  diagonal subgroup under the action in Equation \eqref{GL_n(R) action}.

  Let $T \subseteq \boldZ_2^m,$ let $S \subseteq (\boldR_{>0})^{m},$ and
  let $\Sigma(T,S)$ be as in Definition \ref{Sigma defn}.  For $c \ne 0,$
  let $F_c$ be as in Definition \ref{Fc defn}. 

  If $F_c$ maps $S$ onto $\boldR^d,$ then every Lie algebra in
  $\calS_\Lambda(\boldR)$ is isomorphic to a Lie algebra in
  $\Sigma(T,S).$ If in addition $F_c$ is one-to-one, every Lie algebra
  defined by a Lie bracket in $\calL_\Lambda(\boldR)$ is isomorphic to
  precisely one Lie algebra in $\Sigma(T,S) \cap \calL_\Lambda(\boldR).$
\end{thm}
It is not always  easy to verify  the hypotheses on $F_c$ in the 
  theorem.
  Lemma \ref{main lemma} gives  conditions on
$F_c$  that will force bijectivity in some situations.   Finding a
parametrizing set is less difficult if boundedness and 
algebraic definitions for charts are not required. (See Remark \ref{okay}.)

In Definition
\ref{definition of cross section} we define a subset
$\Delta^p_{\bfa_0}$ of $(\boldR_{>0})^{|\Lambda|}.$ We will use sets
of the form $\Sigma(T,\Delta_{\bfa_0}^p)$ to parametrize isomorphism
classes of products in $\calS_\Lambda(\boldR).$ The sets
$\Sigma(T,\Delta_{\bfa_0}^p)$ are semi-algebraic with compact closure.

\begin{thm}\label{semialgebraic}  
  Let $\Lambda \subseteq \Theta_n$ be an index set of cardinality $m
  >0.$ Let $T \subseteq \boldZ_2^m.$ Let $p$ be a positive 
rational number, let
  $\bfa_0 \in (\boldR_{>0})^m$ and let $\Delta^p_{\bfa_0}$ be as in
  Definition \ref{definition of cross section}.  
  Let $\Sigma(T,\Delta^p)$ be as in Definition \ref{Sigma defn}.  Then
  $\Sigma(T,\Delta^p_{\bfa_0})$ is a semi-algebraic subset of
  $\boldR^m$ with compact closure.
\end{thm}

In the computations in \cite{payne-comp-II}, there are hundreds of
index sets for which the Jacobi Identity for $\calS_\Lambda(\boldR)$
is nontrivial, and it is necessary to solve one or more polynomial
equations in one or more variables for each case.  It is initially
surprising that almost all of the systems of equations are equivalent
to a small number of simple prototypical systems.  This happens because
the system of equations depends on the combinatorics of the index set,
and although there are many index sets, many of those have the same
combinatorial relationships.  All of the 34 $7$-dimensional Lie
algebras in the classification in \cite{payne-comp-II} fall into one
of the five categories listed in the next theorem.  The great majority
of the $8$-dimensional Lie algebras in the classification in
\cite{payne-comp-II} fall into one of the five categories.  Among the
remaining examples in dimension 8, most of those remaining Lie
algebras fall into three more other categories.  There are just a
handful of examples that lie in two last categories.

The  following theorem from \cite{payne-parametrizations} is proved
using the methods developed here.  
It has a hypothesis involving some more combinatorial
definitions. The precise definition of {\em null space spanning} is
given in Section \ref{combinatorics}.  Among Lie algebras of dimension
7 and 8 arising in \cite{payne-comp-II}, all the Lie algebras of
dimension 7 and a vast majority of those of dimension 8 may be
represented by a Lie bracket in $\calS_\Lambda(\boldR)$ where
$\Lambda$ is null space spanning.  We define
 what it means for quadruples to have a common triple  in Definition
\ref{quad-sign}.

\begin{thm}[\cite{payne-parametrizations}]\label{analyze-exhaustive}
Let $\Lambda \subseteq \Theta_n$
be an index set of cardinality $m$ that is null space spanning.  Then

\begin{enumerate}

\item{If $\Lambda$ has exactly one quadruple  of multiplicity two,
    then  $\widetilde \calL_\Lambda(\boldR)$ is finite.}\label{1qm2}

\item{If $\Lambda$ has exactly two quadruples  of multiplicity two,
    then 

\begin{enumerate}

\item{If the quadruples do not have  a common triple, then 
$\widetilde \calL_\Lambda(\boldR)$ is finite. }\label{2qm2a}

\item{If the quadruples have exactly one common triple, then 
$\widetilde \calL_\Lambda(\boldR)$ is one-dimensional.}\label{2qm2b}

\end{enumerate}
}\label{2qm2}

\item{If $\Lambda$ has just one quadruple  of multiplicity three, then 
$\widetilde \calL_\Lambda(\boldR)$ is one-dimensional.}\label{1qm3}

\item{If $\Lambda$ has one quadruple  of multiplicity three,
    and one quadruple  of multiplicity two, and 
the quadruples have exactly one common triple, then 
$\widetilde \calL_\Lambda(\boldR)$ is one-dimensional.}\label{1qm2+1qm3}

\end{enumerate}
\end{thm}
  Furthermore, in \cite{payne-parametrizations} we give formulae depending only on the
combinatorics of the set $\Lambda$ for 
parametrizations of some classes $\widetilde
\calL_\Lambda(\boldR).$

 \section{Preliminaries}\label{preliminaries}

\subsection{Index sets, root vectors, sign vectors, and root matrices}
\label{definitions}

In this section, we give definitions of some algebraic and
combinatorial objects that we will use.

Fix an ordering of $\Upsilon_n.$  Let  $\bfa =
[\alpha_{ij}^k]_{(i,j,k) \in \Lambda}$ be in $\calS_\Lambda(\boldR).$
Define the vector $|\bfa|$ by 
vector $|\bfa|  = [|\alpha_{ij}|]_{(i,j,k) \in \Lambda}$ where
the entries of $|\bfa|$ are listed
in ascending order relative to the fixed ordering of $\Upsilon_n.$
Define
the {\em sign vector} for $\alpha$ to be the vector \[ \sgn(\bfa) =
[\sgn(\alpha_{ij}^k)]_{(i,j,k) \in \Lambda}\] in
$\boldZ_2^{|\Lambda|},$ where $\sgn: \boldR^\times \to  \boldZ_2$ is
the homomorphism to the additive group $\boldZ_2$ defined by
\begin{equation}\label{sgn-def}
\sgn(x) = \begin{cases} 
1 & \text{if $x < 0$}\\
0 & \text{if $x > 0$}
\end{cases} 
\end{equation} 
for nonzero $x \in \boldR.$  Sometimes it will be more convenient to
use the homomorphism $\sign: \boldR^\times \to  \{-1,1\}$ with values
in the multiplicative group $\boldZ_2.$  

 Conversely, a nonempty index set $\Lambda,$ a vector $\bfa =
[a_{(i,j,k)}]_{(i,j,k) \in \Lambda} \in \boldR^{|\Lambda|}_{>0}$ with
positive entries indexed by $\Lambda,$ and a vector $\bfs =
[s_{(i,j,k)}]_{(i,j,k) \in \Lambda} \in \boldZ_2^{|\Lambda|}$
determine a unique bilinear skew-symmetric map $\alpha \in
\wedge^2(\boldR^n)^\ast \otimes \boldR^n.$ The map $\alpha$ is given
by $ \alpha(x_i,x_j) = \sum_{k=1}^n \alpha_{ij}^k x_k, $ where
\[ \alpha_{ij}^k = \sign(a_{(i,j,k)}) |a_{(i,j,k)}| =
(-1)^{\sgn(a_{(i,j,k)})}| a_{(i,j,k)}|.\] (We will repeatedly abuse
notation so that when $s \in \boldZ_2,$ we have $(-1)^s \in \boldR:$
i.e., $(-1)^0 = 1 \in \boldR,$ and $(-1)^1 = -1 \in \boldR.$) The
resulting bilinear map $\alpha$ has $[|\alpha_{(i,j,k)}|]_{(i,j,k) \in
\Lambda} =\bfa$ and sign vector $\bfs.$

For a nonempty index set $\Lambda \subseteq \Upsilon_n$ of cardinality
$m,$ we associate to $\Lambda$ a set of $m$ $n \times 1$ vectors and
an $m \times n$ matrix as follows.  To each triple $(i,j,k) \in
[n]^3,$ we associate the {\em $K$ root vector} or {\em root vector}
$\bfy_{(i,j,k)} = \bfe_i^T + \bfe_j^T - \bfe_k^T$ in the vector space
$K^n,$ where $\{\bfe_i\}_{i=1}^n$ is the standard basis for $K^n$ (as
column vectors).  The {\em $K$ root matrix} for $\Lambda$ is the $m
\times n$ matrix $Y(K)$ whose rows are the root vectors
$\bfy_{(i,j,k)}, (i,j,k) \in \Upsilon_n$ listed in ascending order
relative to the fixed ordering of $\Upsilon_n.$ When we write $Y$
omitting mention of the underlying field, we assume that the
underlying field is $\boldR,$ and when we write $\hat Y,$ we assume
that the underlying field is $\boldZ_2.$ We will give examples of
index sets, sign vectors and root matrices in the coming pages. The
reader may also refer to \cite{payne-09b} or \cite{payne-12a} for more
examples and properties of these objects.

\subsection{Isomorphism classes of nilpotent Lie algebras}\label{isomorphism-section}

In this section we define the subsets $\Sigma(T,\Delta^p)$ of $
\calS_\Lambda(\boldR),$ and we prove Theorem \ref{Fc thm}.  We need to
first define the action of the additive group $K^n$ on the vector
space $K^m$ determined by an $m \times n$ matrix.

\begin{defn}\label{Yaction}
  A $m \times n$ matrix $Y$ with entries in a field $K$ defines an
  action $\rho_Y: K^n \times K^m \to K^m$ of $K^n$ on $K^m$ with
\begin{equation}\label{action def 2}  
   \rho(\bfd,\bfz) = \bfz  + Y \bfd ,
\end{equation}
for $\bfd \in K^n$ and $\bfz \in K^m.$ 
\end{defn}
Orbits of this action are the column space $\col(Y)$ of $Y$ and
translations $\bfz + \col(Y)$ of the column space, where $\bfz$ varies
over $K^m.$ If $K = \boldR$ then each orbit meets $\Null(Y^T)$ exactly
once, so $\Null(Y^T)$ is a simple cross section for the action.

Let $(a_1, \ldots, a_m)$ be a point in $(\boldR_{>0})^m.$ Define the
coordinate-wise logarithm map $\Ln: (\boldR_{>0})^m \to \boldR^m$ by
\begin{equation}\label{L defn}  
\Ln\left( (a_1, \ldots, a_m) \right) = (\ln a_1, \ldots, \ln a_m),
\end{equation}
and let $E$ denote its inverse, the coordinate-wise exponential map, 
\begin{equation}\label{E defn}
  \E \left( (a_1, \ldots, a_m) \right) = (e^{a_1}, \ldots, e^{a_m}).
\end{equation}

The following theorem relates isomorphism classes and orbits of the
action defined by a root matrix.  
\begin{thm}[Theorem 3.8, \cite{payne-index}]\label{action-isomorphism}
Let $\Lambda$ be a subset of $\Theta_n$ of cardinality $m,$ and let
$Y$ and $\hat Y$ denote the real and $\boldZ_2$ root matrices for
$\Lambda$ respectively. Let $\alpha$ and $\beta$ be elements of
$\calS_\Lambda(\boldR) \subseteq \wedge^2(\boldR^n)^\ast \otimes
\boldR^n.$

  Let $\bfa^2 = [(\alpha_{ij}^k)^2]_{(ij,k) \in \Lambda}$ and $\bfb^2
= [ (\beta_{ij}^k)^2]_{(i,j,k) \in \Lambda},$ and let
  \[ \sgn(\bfa) = [\sgn(\alpha_{ij}^k)]_{(i,j,k) \in \Lambda} \quad
\text{and} \quad \sgn(\bfb) = [\sgn(\beta_{ij}^k)]_{(i,j,k) \in
\Lambda}\] be the sign vectors for $\alpha$ and $\beta.$ Let
$\frakg_\alpha$ and $\frakg_\beta$ denote the algebras defined by
$\alpha$ and $\beta$ respectively.

Then the algebras $\frakg_\alpha$ and $\frakg_\beta$ are in the same
$D$ orbit for the action \eqref{GL_n(R) action}  if and only if
\begin{enumerate}
\item{the vectors \[ \Ln(\bfa^2) = [\ln (\alpha_{ij}^k)^2]_{(i,j,k) \in
      \Lambda} \quad \text{and} \quad \Ln(\bfb^2) = [\ln
    (\beta_{ij}^k)^2]_{(i,j,k) \in \Lambda}\] are in the same orbit
    for the action $\rho_Y :\boldR^n \times \boldR^m \to \boldR^m$
    induced by $Y$ as in Definition \ref{Yaction}, and}\label{same Y
    orbit}
\item{the sign vectors $\sgn(\bfa)$ and $\sgn(\bfb)$ are in the same
    orbit for the action $\rho_{\hat Y} : \boldZ_2^n \times \boldZ_2^m
    \to \boldZ_2^m$ as in Definition \ref{Yaction}.}
\end{enumerate}
\end{thm}

\begin{remark}\label{when invariant} One might ask for which $\Lambda
\subseteq \Theta_n$ it is true that isomorphism classes in
$\calS_\Lambda(\boldR)$ are $D$ orbits.  This holds if $\calB$ is an
eigenvector basis for the Nikolayevsky derivation $D^N$ of
$\frakg_\alpha$ and all of the eigenvalues of $D^N$ are positive and
have multiplicity one.  Conveniently, when $\Lambda \subseteq
\Theta_n,$ all members of a stratum $\calL_\Lambda(\boldR)$ have the
same Nikolayevsky derivation (Theorem 3.1, \cite{payne-index}).  All
of the examples of $\Lambda$ that we will present in this work have
Nikolayevsky derivation with positive eigenvalues each of multiplicity
one, with the exception of Example \ref{h5}.

More generally, Proposition 2 and its corollary in 
\cite{bratzlavsky-71} present  conditions that insure that
$D$-orbits are isomorphism classes:  the basis should be 
an eigenbasis for a maximal torus in $\Aut(A)$ 
and the normalizer of $T$ in $\Aut(A)$ should equal  its
centralizer in $\Aut(A).$
A careful reading of the 
proof given there shows that the proof extends to the case that the
base field is $\boldR.$
\end{remark}

Define an action of $\boldZ_2^m = \{0,1\}^m$ on $\boldR^m$
by coordinate-wise multiplication:
\begin{equation}\label{bullet def}
 (s_1, \ldots, s_m) \bullet (a_1, \ldots, a_m) = ((-1)^{s_1} a_1,
\ldots, (-1)^{s_m} a_m).\end{equation} 

\begin{defn}\label{Sigma defn}  Let $\Lambda$ be a subset of
  $ [n]^3$ of cardinality $m >0.$ Let $T \subseteq \boldZ_2^m$ be a
  simple cross section for the $\rho_{\hat Y}$ action, and let $S
  \subseteq (\boldR_{>0})^m.$ Define the subset $\Sigma(T,S)$ of $
  \calS_\Lambda(\boldR)$ by
\[ \Sigma(T,S)  = \{ \bft \bullet  \bfx \, : \, \bft \in T, \bfx \in
S   \} \subseteq \boldR^{|\Lambda|} ,\]
where $\bullet$ denotes the action of $\boldZ_2^m$ on $\boldR^m$
by coordinate-wise multiplication as in \eqref{bullet def}.
\end{defn}

The set $\calS_\Lambda(\boldR)$ is homeomorphic to $(\boldR \setminus
\{0\})^m,$ which is homeomorphic to  
$\boldZ_2^m \times \boldR^m$ via the map $\psi : (\boldR \setminus
\{0\})^m \to \boldZ_2^m \times \boldR^m$
sending $[\alpha_{(i,j,k)}]_{(i,j,k) \in \Lambda}$ to the ordered pair 
\begin{equation}\label{defn psi}
 ([\sgn(\alpha_{(i,j,k)})]_{(i,j,k) \in \Lambda} ,[\ln(|
\alpha_{(i,j,k)}|)]_{(i,j,k) \in \Lambda} ) \in \boldZ_2^m \times
(\boldR_{>0})^m.\end{equation}

The proof of Theorem \ref{action-isomorphism}   is structured as follows.
The action of the diagonal subgroup  $D \cong (\boldR \setminus
\{0\})^n$ of $GL_n(\boldR)$ on $\calS_\Lambda(\boldR)$ is topologically
conjugate by $\psi$  to the action   $\rho_{\hat Y} \times \rho_Y$ of 
$\boldZ_2^n \times \boldR^n$ on  $\boldZ_2^m \times \boldR^m.$  To be precise, let  
$g=\diag(c_1, \ldots, c_n)$
be a diagonal matrix in $GL_n(\boldR),$ and let $\sgn(g)$ be the sign
vector in $\boldZ_2^n$ and let $\Ln(|\bfg|)$ be the vector with $i$th
entry $\ln(|c_i|).$  Then if 
$\alpha = [\alpha_{(i,j,k)}]_{(i,j,k) \in \Lambda} \in \calS_\Lambda(\boldR),$
the product $\beta = g \cdot \alpha$ under the action 
\eqref{GL_n(R) action} has sign vector 
\[ [\sgn(\beta_{(i,j,k)})]_{(i,j,k) \in \Lambda} = 
\rho_{\hat Y}(\sgn(\bfg), [\sgn(\alpha_{(i,j,k)})]_{(i,j,k) \in \Lambda})\] and
\[ [\ln (|\beta_{(i,j,k)}|)]_{(i,j,k) \in \Lambda} = \rho_Y (\Ln(|\bfg|),[\ln
(|\alpha_{(i,j,k)}|)]_{(i,j,k) \in \Lambda}).\]
   Therefore, the map $\psi$ in \eqref{defn psi} induces
 a one-to-one correspondence between $D$ orbits in
$\calS_\Lambda(\boldR)$ and
 $\rho_{\hat Y} \times \rho_Y$ orbits in
$\boldZ_2^m \times \boldR^m.$  If $T$ and $\Ln(S)$
 are simple cross
sections for the $\rho_{\hat Y}$ and $\rho_Y$ actions respectively,
then $T \times \Ln(S)$ is a simple cross section for the $\rho_{\hat Y}
\times \rho_Y$ action.  But $T \times \Ln(S) \subseteq \boldZ_2^m \times
(\boldR_{>0})^m$ is homeomorphic to 
$\psi^{-1}(T, \Ln(S)) = \Sigma(T,S) \subseteq (\boldR \setminus
\{0\})^m.$   This yields the following corollary to Theorem \ref{Fc thm}. 
\begin{cor}\label{cross section cor}
  Let $\Lambda$ be a subset of $\Theta_n$ of cardinality $m >0,$ 
and let $Y$ and $\hat Y$ denote the real and $\boldZ_2$ 
root matrices for $\Lambda.$ Let $T$ be a simple cross section for the
  action $\rho_{\hat Y}$ induced by $\hat Y$ as in Definition
  \ref{Yaction}, and let $S \subseteq (\boldR_{>0})^{m}$ be
  such that $\Ln(S)$ is a simple cross section for the action
  $\rho_{Y}$ induced by $Y$ as in Definition \ref{Yaction}.

  Then every algebra in $\calS_\Lambda(\boldR)$ is represented at
least once in the set $\Sigma(T,S).$ If isomorphism classes of
algebras (Lie algebras) in $\calS_\Lambda(\boldR)$ are $D$ orbits,
then every algebra (Lie algebra) in $\calS_\Lambda(\boldR)$ is
isomorphic to precisely one algebra defined by an element of
$\Sigma(T,S).$

  Furthermore, if the natural map from $\calS_\Lambda(\boldR)$ to
  $\Sigma(T,S)$ is continuous, $\Sigma(T,S)$ is homeomorphic to
  $\widetilde \calS_\Lambda(\boldR).$
\end{cor}

We illustrate the corollary with two examples.

\begin{example}\label{l4} 
Let $\Lambda = \{ (1,2,3), (1,3,4)\} \subseteq \Theta_4.$ The set
$\calS_\Lambda(\boldR)$ consists of all real anticommutative
nonassociative algebras spanned by a fixed basis $\calB = \{ x_1, x_2,
x_3, x_4 \}$ with the product determined by the relations
\begin{equation}\label{l4 defn} [x_1,x_2] = \alpha_{12}^3 \, x_3,
\qquad [x_1,x_3] = \alpha_{13}^4 \, x_4,
\end{equation} where $\alpha_{12}^3$ and $\alpha_{13}^4$ are nonzero
real numbers.  By Remark \ref{when invariant}, isomorphism classes in
$\calS_\Lambda(\boldR)$ are $D$ orbits.

    The $\boldR$-root matrix $Y$ and the $\boldZ_2$-root matrix $\hat
Y$ are given by
\[  Y = \begin{bmatrix} 1 & 1 & -1 & 0 \\ 
1 & 0 & 1 & -1 \end{bmatrix} \quad \text{and} \quad
 \hat Y = \begin{bmatrix} 1 & 1 & 1 & 0 \\ 
1 & 0 & 1 & 1 \end{bmatrix} .
\] 
Because the rank of $Y$ is two, the action $\rho_Y : \boldR^4 \times
\boldR^2 \to \boldR^2$ induced by $Y$ as in Definition \ref{Yaction}
is transitive, and $\Null(Y^T) = \{(0,0)\} \subseteq \boldR^2$ is a
simple cross section for the action.  The rank of $\hat Y$ is also
two, so the action $\rho_{\hat Y}: \boldZ_2^4 \times \boldZ_2^2 \to
\boldZ_2^2$ induced by $\hat Y$ as in Definition \ref{Yaction} is also
transitive with simple cross section $T = \{ (0,0)\} \subseteq
\boldZ_2^2.$

If we let $\Delta = \{ (1,1)\},$ then the simple cross section
$\{(0,0)\}$ for the $\rho_Y$ action is equal to $\Ln(\Delta).$ The set
$\Sigma(T,\Delta)$ as defined in Definition \ref{Sigma defn} is
 \begin{equation}\label{sigma t delta l4} \Sigma(T,\Delta) = \{  (0,0)
   \bullet (1,1)\} = \{ (1,1) \} \subseteq
   \calS_\Lambda(\boldR).\end{equation}

By Corollary \ref{cross section cor}, 
all Lie algebras in $\calL_\Lambda(\boldR)$ 
are isomorphic to the Lie algebra in $\Sigma(T,\Delta).$ 
\end{example}

\begin{example}\label{h5}
Let $\Lambda = \{ (1,2,5), (3,4,5)\} \subseteq \Theta_5.$ Products in
$\calS_\Lambda(\boldR)$ are determined by
\begin{equation}\label{h5 defn} [x_1,x_2] = \alpha_{12}^5 \, x_5,
\qquad [x_3,x_4] = \alpha_{34}^5 \, x_5,
\end{equation} relative to basis $\calB = \{ x_i \}_{i=1}^5,$ where
$\alpha_{12}^5, \alpha_{13}^5 \ne 0.$

    The $\boldR$-root matrix $Y$ and the $\boldZ_2$-root matrix $\hat
Y$ are given by
\[  Y = \begin{bmatrix} 1 & 1 & 0 & 0& -1  \\ 
0 & 0 & 1 & 1 & -1 \end{bmatrix} \quad \text{and} \quad
 \hat Y = \begin{bmatrix} 1 & 1 & 0 & 0 & 1 \\ 
0 & 0 & 1 & 1 & 1 \end{bmatrix} .
\] 

As in the previous example, $\Null(Y^T) = \{(0,0)\} \subseteq
\boldR^2$ is a simple cross section for the $\rho_Y$ action, and $T =
\{ (0,0)\} \subseteq \boldZ_2^2$ is a simple cross section for the
$\rho_{\hat Y}$ action.  By Corollary \ref{cross section cor}, all
algebras in $\calL_\Lambda(\boldR)$ are isomorphic to the algebra in
$\Sigma(T,\Delta) = \{1,1\}.$
\end{example}

Our basic strategy for describing all the Lie algebras in a stratum 
$\calS_\Lambda(\boldR)$ is as follows.
\begin{enumerate}
\item{Find sets $T \subseteq \boldZ_2^m$ and $S \subseteq
    (\boldR_{>0})^m$ so that $T$ and $\Ln(S)$ are cross sections
    for the $\rho_{\hat Y}$ and $\rho_Y$ actions, respectively.  Let
    $\Sigma(T,S)$ be as in Definition \ref{Sigma defn}.}
\item{If $S$ has dimension $d \ge 2,$ find a simple
    parametrization of the set $\Ln(S),$ if possible, in terms of
    parameters $t_1, t_2, \ldots, t_d.$ This will involve using the
    vectors $\bfw(m_1,m_2,m_3,m_4)$ defined in Definition
    \ref{w-defn}.  These vectors have the advantage of being linear
    combinations of just 4 basis vectors, with coefficients of $-1$
    and $1.$ }
\item{Show that every Lie algebra in $\calS_\Lambda(\boldR)$ is isomorphic to
    precisely one Lie algebra in $\Sigma(T,S)$ by 
    using Theorem \ref{Fc thm} or Corollary \ref{cross section cor}. }
\item{Use Theorem \ref{ji-quads-signs} to solve the Jacobi
    Identity to determine which elements of $\Sigma(T,S)$ are Lie
    algebras.  Sometimes, this may be done in two steps:
\begin{enumerate}
\item{Find the values of $|\alpha_{ij}^k|$  by solving a system of polynomial equations
  in  $t_1, t_2, \ldots, t_d.$ }
\item{Find which possible signs may be assigned to the structure
    constants so that the Jacobi Identity is satisfied.  }
\end{enumerate}}
\end{enumerate}

For many $\Lambda$'s, the calculations described above are
essentially identical.   As we proceed, we
will repeatedly return to prototypical examples of each 
type covered by
 Theorem 
\ref{analyze-exhaustive}:
\begin{itemize}
\item{Analysis of a stratum $\calS_\Lambda(\boldR)$ as in 
Part \ref{1qm2} of Theorem \ref{analyze-exhaustive}
    is worked out in Examples \ref{1qm2-a} and \ref{1qm2-b}.}
\item{Analysis of a set  $\calS_\Lambda(\boldR)$ as in Part \ref{2qm2} 
of Theorem \ref{analyze-exhaustive}
    is in Example \ref{2qm2b-a}.}
\item{Analysis of a set $\calS_\Lambda(\boldR)$ as in Part \ref{1qm3} of
 Theorem \ref{analyze-exhaustive}
    is worked out
  in Examples \ref{1qm3-a}, \ref{1qm3-b}, \ref{1qm3-c} and
    \ref{1qm3-d}.}
\item{An example as in Part \ref{1qm2+1qm3} 
of Theorem \ref{analyze-exhaustive}    is analyzed in Examples 
\ref{example-1a}, \ref{example-1b}, \ref{example-1c}, 
\ref{example-1d}, and \ref{example-1e}.}
\end{itemize}

\section{Cross sections}\label{cross sections}

We would like to generalize the approach used in Example \ref{l4} to
apply to cases in which the simple cross section $\Ln(\Delta)$ is not
finite, as it was in Example \ref{l4}.  In that example, we found that
all algebras in $\calS_\Lambda(\boldR)$ were isomorphic to the one
with structure constants defined by
\[ \{ (1,1)\} =\Sigma( \{(0,0)\}, \{(1,1)\}) = \Sigma( \Null(\hat
Y^T), \E(\Null(Y^T)))  .\]

\begin{remark}\label{okay}
There is an obvious generalization.  When $Y$ does not have
maximal rank, we may use an analogous definition to obtain a set of
form $\Sigma(T,\Delta)$ where $T = \Null(\hat Y^T)$ and $\Delta =
\E(\Null(Y^T)).$ However, this set is noncompact when $\Null(Y^T)$ is
infinite, and requires exponential functions for its parametrization.
Yet, this may be a useful simple cross section if boundedness 
and algebraic charts are not required. 
\end{remark}

Now we define the type of set that will be our simple cross section
for the $\rho_Y$ action in Equation \eqref{Yaction}.
\begin{defn}\label{definition of cross section}
Let $Y$ be an $m \times n$ root matrix. 
Let $\bfa_0$ be a point in $(\boldR_{>0})^m$ and let 
\begin{align*}
  \Delta_{\bfa_0} &=  \{ \bfa_0 + \bfw \, : \,  \bfw \in \Null(Y^T) \}
\bigcap  \, \boldR_{>0}^m , \, \text{so}\\
\Ln (\Delta_{\bfa_0} ) 
 &= \{ \Ln  (\bfa_0 + \bfw)  \, : \, \bfw  \in
\Null(Y^T), \bfa_0 + \bfw \in  (\boldR_{>0})^m  \}.
\end{align*} 
For any $p \ne 0,$ let 
\begin{align}\label{defn Delta p}
  \Delta^p_{\bfa_0} &= \{  \Exp (p\bfa) \, : \, \bfa \in \Ln(\Delta_{\bfa_0})  \} \\
  &= \{ (a_1^p, a_2^p, \ldots, a_m^p) \, : \, (a_1, \ldots, a_m) \in
  \Delta_{\bfa_0} \} \notag 
\end{align} We say that the point
$\bfa_0$ is the {\em center point} of $\Delta_{\bfa_0}^p.$
\end{defn}

\begin{remark}\label{Delta p}  Because  
$ [\ln (|\alpha_{ij}^k|^p)]_{(i,j,k) \in \Lambda} = p [\ln
 ( |\alpha_{ij}^k|) ] _{(i,j,k) \in \Lambda},$ the condition in Part \eqref{same
    Y orbit} of the Theorem \ref{action-isomorphism}
 holds if and only if the vectors $ [\ln
 ( |\alpha_{ij}^k|^p)]_{(i,j,k) \in \Lambda}$ and $ [\ln
 ( |\beta_{ij}^k|^p)]_{(i,j,k) \in \Lambda}$ are in the same orbit for
  the action $\rho_Y$ for any nonzero $p.$ Therefore, for any $p \ne
  0,$ $\Ln (\Delta^p)$ is a simple cross section for the $\rho_Y$ action
  if and only if $\Ln (\Delta)$ is.   
\end{remark}

Our next goal is to show that under suitable hypotheses, the set $\Ln
(\Delta_{\bfa_0}^p)$ is a simple cross section for the
$\rho_Y$-action.  First we need an elementary lemma.
 
\begin{lemma}\label{pi lemma}  
  Let $Y$ be an $m \times n$ matrix over $\boldR.$ Let $\rho_Y :
\boldR^n \times \boldR^m \to \boldR^m$ be the action defined in
Definition \ref{Yaction}.  Suppose that $\Null(Y^T)$ is nontrivial
with basis $\calD = \{ \bfw_1, \ldots, \bfw_d\}.$ Define the map
$\pi_Y: \boldR^m \to \boldR^d$ by
\begin{equation}\label{pi}  \pi_Y: \bfv \mapsto
 (\bfv \cdot \bfw_1, \ldots, \bfv \cdot \bfw_d),\end{equation}
where $\cdot$ denotes the standard dot product of two vectors in $\boldR^m.$

Let $S $ be a nonempty subset of $\boldR^{m}.$ The set $S$ meets each
orbit of $\rho_Y$ at least once if and only if the restriction of
$\pi_Y$ to $S$ is surjective, and $S$ is a simple cross section for
the action if and only if the restriction of $\pi_Y$ to $S$ is a
bijection between $S$ and $\boldR^m.$
\end{lemma}

\begin{proof}  
The orbits of $\rho_Y$ are the sets $\bfz + \col(Y),$ where $\bfz \in
\boldR^m.$ Because the vector space $\boldR^m$ is the orthogonal
direct sum of the null space $\Null(Y^T)$ and the column space
$\col(Y),$ we may always choose a unique $\bfz$ in $\Null(Y^T)$ in
each orbit.  Let $\calB = \{ \bfv_1, \ldots, \bfv_d \}$ be an
orthonormal basis for $\Null(Y^T).$ 
 Let $p: \boldR^m \to\boldR^d$ denote the map
\[ p(\bfv) = (\bfv \cdot \bfv_1, \ldots, \bfv \cdot \bfv_d),\] which
is orthogonal projection from $\boldR^m$ to $\Null(Y^T),$ relative to
the basis $\calB.$ It follows that the set $S$ meets the orbit $\bfz +
\col(Y), $ with $\bfz \in \Null(Y^T),$ if and only if the coordinate
vector $[\bfz]_\calB$ of $\bfz$ with respect to $\calB$ is in $p(S).$
Hence $S$ meets each orbit at least once if and only if $p|_S$ is
surjective, and $S$ meets each orbit at most once if and only if
$p|_S$ is one-to-one.

Since $\{ \bfw_1, \ldots, \bfw_d \}$ is a linearly independent set,
there exists an invertible linear transformation
$T: \boldR^m \to \boldR^m$ such that $p \circ T = \pi_Y.$ 
Hence $p|_S$ is bijective if and
only if $\pi_Y|_S$ is bijective, and $p|_S$ is onto if and only if
$\pi_Y|_S$ is onto. 
\end{proof}

Now we prove Theorem \ref{semialgebraic}, 
showing that $\Sigma(T,\Delta^p_{\bfa_0})$ is semi-algebraic subset
  of $\boldR^m$ with compact closure.

\begin{proof}
Let $\{ \bfw_1, \ldots, \bfw_d\}$ be a basis for $\Null(Y^T).$
Clearly $\Delta_{\bfa_0}$ is semi-algebraic, as the defining condition
\[ \bfx = (x_j) = \bfa(t_1, \ldots, t_d) = \bfa_0 + \sum_{i=1}^d t_i
\bfw_i > 0 \] for $\Delta_{\bfa_0}$ is equivalent to requiring that
$x_j > 0 $ for all $j=1, \ldots, m$ and that $Y^T (\bfx - \bfa_0) =
\bfzero.$ Hence, $\Delta_{\bfa_0}$ is a semi-algebraic subset of 
$\boldR^m.$  

Let $p=r/s,$ with $r,s \in \boldZ_{>0}.$   Then 
$\bfx \in \Delta_{\bfa_0}^p$ if and only if 
$x_i^r = y_i^s$  for $\bfy \in \Delta_{\bfa_0}.$   This means
that $\Delta^p_{\bfa_0}$ is the projection of the semi-algebraic set 
\[ \{ (\bfx,\bfy) \, : \, \bfy \in \Delta_{\bfa_0}, x_i^r = y_i^s
\enskip \text{for all $i$}\}\]
onto the first factor.  
The projection of a semi-algebraic set is semi-algebraic, hence
$\Delta^p_{\bfa_0}$ is semi-algebraic.

To see that $\Delta^p_{\bfa_0}$ has compact closure, by continuity of
the map $x \mapsto x^p,$ it suffices to
show  that  $\Delta_{\bfa_0}$ has compact closure. 
Since each row of $Y$ is of the form $\bfe_i + \bfe_j - \bfe_k$ for
$i < j < k,$ $Y [1]_{n \times 1} = [1]_{m \times 1},$ where
$[1]_{k \times 1}$ denotes the $k \times 1$ vector with all entries
$1.$   As $[1]_{m
  \times 1}$ is in the column space of $Y,$ 
$\Null(Y^T)$ is contained in the hyperplane $[1]_{m \times 1}^\perp.$ 
Therefore, $\Delta_{\bfa_0}$ is contained in the intersection of 
$\bfa_0 +  [1]_{m \times 1}^\perp$ and $(\boldR_{>0})^m,$ a bounded set.
Hence,   $\Delta^p_{\bfa_0}$ is bounded.   

Now consider $\Sigma(T,\Delta^p).$   Let $T = \{ \bfs_1, \ldots,
\bfs_k \}.$   
The set $\Sigma(T,\Delta^p)$  is the union of the $k$ disjoint 
sets $\{ \bfs_i \} \bullet \Delta^p,$
where $i=1, \ldots, k.$  But since coordinate-wise multiplication is a
linear map, the sets $\{ \bfs_i \} \bullet \Delta^p$ are all
semi-algebraic.  The union of the semi-algebraic sets is
semi-algebraic,
 making $\Sigma(T,\Delta^p)$  algebraic.  Clearly the finite union of
  bounded sets is bounded, so
  $\Sigma(T,\Delta^p)$ has compact closure.   
\end{proof}

Before we proceed with the proof of Theorem \ref{Fc thm} 
we give an example of an application of Lemma \ref{pi lemma}.
\begin{example}\label{1qm2-a}
Let $n=7,$ let 
\[ \Lambda = \{ (1,2,4), (1,3,5), (1,5,6), (2,4,6), (2,5,7),(3,4,7)\}\]
and let $Y$ be the $6 \times 7$ real root matrix for $\Lambda.$ 
The vector 
 \[ \bfw_1 =  (1,-1,0,0,-1,1)^T \] 
is a basis for $\Null(Y^T).$
Let $\bfa_0 = (1,1,1,1,1,1)^T,$ and for  $s \in \boldR$ let
\[ \bfa(s) = \bfa_0 + s\bfw_1 = (1,1,1,1,1,1)^T +
s(1,-1,0,0,-1,1)^T. \] 
The vector $\bfa(s)$ is in $(\boldR_{>0})^6$ if
and only if $s \in (-1,1).$ The set $\Delta_{\bfa_0}$ as in Definition
\ref{definition of cross section} is
\begin{align*}
 \Delta_{\bfa_0} &= \{  \bfa(s)  \, : \, -1 < s < 1 \} \\ 
&=
 \{   (1,1,1,1,1,1)^T + s(1,-1,0,0,-1,1)^T \, : \, -1 < s < 1 \} .
\end{align*}
For $s \in (-1,1),$
\[ \Ln(\bfa(s)) =   (\ln(1 +s),\ln(1 -s),0,0,\ln(1 -s),\ln(1+s)), \]
and 
\[
 \Ln(\Delta_{\bfa_0}) = \{ \Ln (\bfa(s)) \, : \,  -1 < s < 1  \} .\]
The map $\pi_Y : \boldR^6 \to \boldR$ as in
  Lemma \ref{pi lemma} is expressed in terms of the parameter $s$ by
 \[ \pi_Y : \Ln( \bfa(s)) \mapsto \Ln (\bfa(s)) \cdot (1,-1,0,0,-1,1) = 2 \,
 \ln \left( \frac{ 1 +s}{1-s} \right).\] 
The function $s \mapsto \pi_Y(\Ln( \bfa(s)))$ is monotonically
 increasing on $(-1,1)$ with $ \lim_{s \to -1} \phi(s) = -\infty$ and $ \lim_{s
   \to -1} \phi(s) = \infty.$ Hence it is a bijection from $(-1,1)$ to
 $\boldR.$ Therefore, by Lemma \ref{pi lemma}, $\Ln(\Delta_{\bfa_0})$ is
 a simple cross section for the action $\rho_Y.$ 

 Now consider the action $\rho_{\hat Y} : \boldZ_2^7 \times \boldZ_2^6
 \to \boldZ_2^6$ defined by the $\boldZ_2$ root matrix $\hat Y.$ It
 can be shown that the rank of $\hat Y$ over $\boldZ_2$ is $5$ and that
\[  T = \{ (0,0,0,0,0,0)^T, (0,0,0,0,0,1)^T \} \]
is a simple cross section for the action.  
The set $\Sigma(T,\Delta_{\bfa_0})$ as in Definition \ref{Sigma defn} is 
 \begin{equation}\label{1qm2-a rep set} 
\Sigma(T,\Delta_{\bfa_0}) =  \{   (1 + s,1 - s,1,1,1-s, \pm (1+s))^T \, : \, -1 < s < 1 \} .\end{equation}
Note that this is a semi-algebraic subset of $\boldR^6$ with compact
closure, with two charts, each parametrized linearly by $s.$  

Isomorphism classes are $D$ orbits by Remark \ref{when invariant}.
By Corollary \ref{cross section cor}, every Lie algebra in
$\calS_\Lambda(\boldR)$ is isomorphic to precisely one algebra whose
structure constants are encoded by a point in $\Sigma(T,\Delta) \cap
\calL_7(\boldR).$ In Example \ref{1qm2-b} we shall determine which algebras in
$\Sigma(T,\Delta)$ satisfy the Jacobi Identity.
\end{example}

When $S$ is one-dimensional, one may use univariate calculus as in the
previous example to show that the map $\pi_Y : S \subseteq \boldR^m
\to \boldR^d$ as in Equation \eqref{pi} of Lemma \ref{pi lemma} is a
bijection.  It is more difficult to show that a map into a
higher-dimensional Euclidean space $\boldR^d$ is a bijection.  

\begin{defn}\label{Fc defn}   
Let $\Lambda$ be an index set of cardinality $m > 0$ whose
root matrix $Y$ has $\dim \Null(Y^T) = d > 0.$
Let $S$ be a $d$-dimensional subset of $\boldR^m,$ 
and let   $\pi_Y : \boldR^m
\to \boldR^d$ be a map of the form in
 Equation \eqref{pi} of Lemma \ref{pi lemma}.
Let $\bfa: D \to S$ be a parametrization of $S;$ that is 
a bijection  sending $(t_1, \ldots, t_d)$ in
a subset $D$ of $\boldR^d$   to $\bfa(t_1, \ldots, t_d)$ in $S.$   

 Define for $c \in \boldR,$ the map $F_c: D \to \boldR^d$ by
\begin{equation}\label{Fc defn equation} 
F_c(t_1, \ldots,t_d) = c (\pi_Y \circ \Ln \circ \bfa )(t_1, \ldots,t_d),
\end{equation} 
where $\Ln$ is the coordinate-wise logarithm map as
defined in \eqref{L defn}.
\end{defn}

This definition and its utility are illustrated in the next
example. 
For a field $K,$ we use $i: K^n \to \boldP_n(K)$ to denote the map
embedding $K^n$ in $n$-dimensional projective space, sending $(x_1,
\ldots, x_n)$ in $K^n$ to $[x_1 : \cdots : x_n : 1]$ in $\boldP_n(K).$
We use $P_n(\boldR)_{\ge 0}$ to indicate the subset
\[ P_n(\boldR)_{\ge 0} = \{ [y_1 : \cdots : y_{n+1}] \, : \, y_1,
\ldots, y_{n+1} \ge 0 \} \] of $P_n(\boldR),$ and we define
$P_n(\boldR)_{>0} $ analogously.

\begin{example}\label{1qm3-a} 
Let $n=7,$ and let
\[ \Lambda= \{(1,2,4), (1,3,5), (1,4,6), (1,6,7), (2,3,6), (2,5,7),
(3,4,7) \}.\] 
The real $7 \times 7$ root matrix $Y$ for $\Lambda$ has
rank $5$ and the vectors
\begin{align*} 
\bfw_1 &=(0,1,0,-1,-1,1,0)^T, \text{and} \\ 
\bfw_2 &= (1,0,0,-1,-1,0,1)^T
\end{align*} 
span $\Null(Y^T).$ The $\boldZ_2$-root matrix $\hat Y$
also has rank $5$ with
\[ T = \myspan_{\boldZ_2}\{ (0,0,0,0,0,0,1)^T, (0,0,0,0,0,1,0)^T \}\]
being a simple cross section for the $\rho_{\hat Y}$ action on
$\boldZ_2^7.$

Let $\bfa_0 = (1,2,1,1,1,2,1)^T.$ We use  $\bfa_0$ as
our center point for $\Delta_{\bfa_0},$ rather than $\bfb_0 =
(1,1,1,1,1,1,1)^T,$ because the algebra defined by $\bfa_0$ satisfies
the Jacobi Identity, whereas an algebra defined by $\bfb_0$ is not a
Lie algebra.  For $(s,t)$ in $\boldR^2,$ let
\begin{align*} \bfa(s,t) &= \bfa_0 + s \bfw_1 + t \bfw_2 \\
&=(1+t,2+s, 1,1-s-t, 1-s-t, 2+s, 1+t )^T .
\end{align*} The entries of $\bfa(s,t)$ are all positive if and only
if $(s,t)$ is in the triangle interior
\[ D = \{ (s,t) \, : \, s > -2, t > -1, \text{and} \, s+t < 1 \}
\subseteq \boldR^2.  \] Here, the set $\Delta_{\bfa_0} \subseteq
\boldR^7$ as in Definition \ref{definition of cross section} is given
by
\[ \Delta_{\bfa_0} = \{ \bfa(s,t) \, : \, (s,t) \in D\},\] and $
\Ln(\Delta_{\bfa_0})$ is the set of all points of form
\begin{tiny}
\[ \left(\ln(1+t), \ln(2+s), 0, \ln(1-s-t), \ln(1-s-t), \ln(2+s), \ln(
1+t) \right)^T \]
\end{tiny} with $(s,t)$ in $D.$

In order to apply Lemma \ref{pi lemma} to show that
$\Ln(\Delta_{\bfa_0})$ is a simple cross section for the $\rho_Y$
action we need to show that the map $F_1=  \pi_Y \circ \Ln \circ \bfa: D \to \boldR^2
$ given by
\begin{align*} 
F_1(\bfa(s,t))  &= 
\left( \Ln(\bfa(s,t)) \cdot \bfw_1, \Ln(\bfa(s,t)) \cdot \bfw_2 \right) \\ 
&= (2 \ln(2+s) - 2 \ln(1-s-t), 2 \ln(1+t) - 2 \ln(1-s-t)) \\ 
&= \left( 2 \ln \left( \frac{2+s}{1-s-t} \right), 
2 \ln \left( \frac{1 +t}{1-s-t}
\right) \right)
\end{align*} is bijection.  Compose this map with coordinate-wise
exponentiation to get the map $G_1 = E \circ F_1:$ 
\[ G_1(s,t) = \left( \left(
    \frac{2+s}{1-s-t} \right)^2 , \left( \frac{1 +t}{1-s-t} \right)
\right)^2.\] 
The map $\pi_Y \circ \bfa$ is a bijection from $D$ to
$\boldR^2$ if and only if $E \circ \pi_Y \circ \Ln \circ \bfa$ has
image $(\boldR_{>0})^2.$ This in turn is true if and only if the map
\[ G_{1/2}(s,t) = (E \circ F_{1/2})(s,t) = \left(
\frac{2+s}{1-s-t}, \frac{1 +t}{1-s-t} \right),\]
where $F_{1/2}$ is as defined in
Definition \ref{Fc defn}, has image $(\boldR_{>0})^2.$

We compose $G_{1/2}$ with the imbedding $i$ of $\boldR^2$ into
$\boldP_2(\boldR)$ and obtain
\[ ( i \circ G_{1/2}) (s,t) = [t + 1 : s + 2: - s- t + 1 ].\] In order
to show that $\pi_Y : D \to \boldR^2$ is bijective it suffices to show
that $i \circ G_{1/2}$ is a bijection from $D$ onto $P_2(\boldR)_{>
  0}.$

Clearly $i \circ G_{1/2}$ is one-to-one as $i \circ G_{1/2}$ extends
to the automorphism
\[ (s,t,u) \mapsto [ t+u : s + 2u : - s- t + 2u ]\] of
$\boldP_2(\boldR).$ The boundary of $D$ is mapped by $i \circ G_{1/2}$
onto the boundary of $P_2(\boldR)_{\ge 0},$ so by continuity, $P$ 
sends $D$ onto $P_2(\boldR)_{> 0}.$ Hence $\pi_Y$ is a bijection and by
Lemma \ref{pi lemma}, $\Ln(\Delta_{\bfa_0})$ is a simple cross section
for the $\rho_Y$ action.

By Remark
\ref{when invariant}, isomorphism classes in $\calS_\Lambda(\boldR)$ 
are $D$ orbits. By Theorem \ref{Fc thm}, every Lie algebra in
$\calS_\Lambda(\boldR)$ is isomorphic to precisely one Lie algebra
represented by an element in the intersection parametrizing set
\begin{tiny}
 \[ \Sigma(T, \Delta_{\bfa_0}) = \{ \left(1+t,2+s, 1, 1-s-t, 1-s-t,
\pm(2+s), \pm( 1+t) \right)^T \, : \, s,t \in D\} \]
\end{tiny} and $\calL_7(\boldR).$ By the same theorem,
we can use a different simple cross section for the
$\rho_Y$ action to define $\Sigma.$  
 By Remark \ref{Delta p}, the set $\Sigma(T, \Delta_{\bfa_0}^{1/2})$
 of points of form   
\[
 \left((1+t)^{1/2},(2+s)^{1/2}, 1,(1-s-t)^{1/2}, (1-s-t)^{1/2},
   \pm(2+s)^{1/2}, \pm( 1+t)^{1/2} \right)^T  ,\]
with $(s,t) \in D,$
is also a parametrizing set for $\calS_\Lambda(\boldR).$ We have not yet considered
the issue of the Jacobi Identity; we will return to this in Example
\ref{1qm3-d}.
\end{example}

Unfortunately, the method for showing bijectivity in the previous
example does not generalize broadly. Instead, one may use the
following generalization of 
Hadamard's Global Inverse Function Theorem to show injectivity.  

\begin{thm}[\cite{gordon-72}]\label{gordon}
  Let $M_1$ and $M_2$ be connected, oriented
$n$-dimensional smooth manifolds of class $C^2,$ with $M_2$ simply
connected.      
A $C^1$ map $f: M_1 \to M_2$
is a diffeomorphism if and only if it is proper and the Jacobian $\det
(\partial f_i / \partial x_j)$ never vanishes.
\end{thm}
Recall that $f: M_1 \to M_2$ is proper if for all compact $K \subseteq
M_2,$ the preimage $f^{-1}(K)$ in $M_1$ is compact.

In the next example, we show how to use this theorem to show
that $\Delta$ is a global cross section.  
\begin{example}\label{example-1a} 
  Let $\Lambda \subseteq \Theta_7$ be the set
\begin{align*} 
\Lambda = \{ (1,2,3),& (1,3,4), (1,4,5),  (1,5,6), (1,6,7), \\ 
& (2,3,5),  (2,4,6), (2,5,7), (3,4,7) \}.
\end{align*} 
By Remark \ref{when invariant}, 
isomorphism classes in  
$\calS_\Lambda(\boldR)$ are $D$ orbits.

The set 
$T = \myspan_{\boldZ_2}\{  \bfe_7, \bfe_8, \bfe_9 \}$
is a simple cross section for the $\rho_{\hat Y}$ action. 
The vectors 
\begin{alignat*}{3}
\bfw_1  & =  (0,-1,0,   1,0,1,  -1,0,0)^T\\
\bfw_2 &= (-1,0,1,   0,0,0,   0,1,-1)^T \\
\bfw_3 &= (-1,0,0,   0,1,0,   1,0,-1)^T
\end{alignat*} 
span $\Null(Y^T).$ 

The algebra defined by $\bfa_0 = (1,1,1,1,1,2,2,1,1)^T$
satisfies the Jacobi Identity.  For $s,t,u \in \boldR,$ let 
\begin{equation}\label{a(s,t,u)}
 \bfa(s,t,u) = \bfa_0 + s\bfw_1 + t\bfw_2 + u\bfw_3 ,\end{equation}
and let
\[ \Delta_{\bfa_0} = \{ \bfa(s,t,u) \, : \, (s,t,u) \in \boldR \} \cap
(\boldR_{>0})^9.\]
The domain $D$ for the parametrization $\bfa:  D \to \boldR^9$ of
$\Delta_{\bfa_0}$  is the bounded convex set defined by
 the five inequalities
\[ 
 1 - t- u > 0, 1 - s> 0,   1 + t> 0, 
 1 +s> 0,  
 1+u> 0. \]

For $(s,t,u) \in D,$
 $F = \pi_Y \circ \Ln \circ \bfa(s,t,u)$ is equal to
\[
\left(\ln \left( 
\frac{(1+s)(2+s)}{(1-s)(2-s+u)}\right),
\ln \left( \frac{(1+t)^2}{(1-t-u)^2}\right),
\ln \left(
  \frac{(1+u)(2-s+u)}{(1-t-u)^2}\right)\right) .\]
Define the map $G: D \to \boldR^3$ by
\begin{multline}\label{G}  G(s,t,u) =  (E \circ F )(s,t,u) =
  \\ 
\left(
\frac{(1+s)(2+s)}{(1-s)(2-s+u)},
\frac{(1+t)^2}{(1-t-u)^2},
  \frac{(1+u)(2-s+u)}{(1-t-u)^2}\right).
\end{multline}
If we embed $D$ into $P_3(\boldR)$  using $i: \boldR^3 \to P_3(\boldR),$
the map $i \circ G:  i(D) \to P_3(\boldR)$ may be expressed as 
\[ (s,t,u,v) \mapsto 
[p_1(s,t,u):p_2(s,t,u):p_3(s,t,u):p_4(s,t,u)]\] where the
polynomials $p_1, p_2, p_3$ and $p_4$ are
\begin{align*}
p_1(s,t,u) &=  (v+s)(2v+s)(v-t-u)^2\\
p_2(s,t,u) &= (v+t)^2(v-s)(2v-s+u)\\
p_3(s,t,u) &= (v+u)(2v-s+u)^2(v-s) \\
p_4(s,t,u) &= (v-s)(2v-s+u)(v-t-u)^2 .
\end{align*}
Note that this map is not defined at $(s,t,u,v)=(1,-1,2,1),$ hence can
not be extended to the boundary of $i(D).$ 

However, the map $G$ is proper.  To show this, we need to show that
if a sequence of points $\bfx_i$ approaches $\partial D,$ the sequence
$G(\bfx_i)$ approaches $\partial (\boldR_{>0})^3.$  As $\bfx_i \to \partial
D,$
some numerator or denominator of a coordinate function
 in Equation \eqref{G} must go to zero.  Then that coordinate will go
 to zero or infinity, unless both the numerator and denominator go to
 zero as $\bfx_i \to \partial
D.$  There are three cases to consider, one for each coordinate
function.  If the numerator and denominator of 
$G_1$  go to zero simultaneously,
then $s \to 1$ and $s \to -1,$ a contradiction. 
 (Note that $2 -s + u = (1-s) + (1+u)$
can only go to zero if both $1-s$ and $1+u$ go to zero.)  If the
numerator and denominator of $G_2$ both go to zero, then $u \to 2,$ so
$1 + u$ does not go to zero.  This implies that $G_3 \to \infty.$  If 
the  numerator and denominator of $G_3$ both go to zero, then $t \to
2.$  Then $G_2 \to \infty.$   Thus, as  $\bfx_i \to \partial
D,$ $G(\bfx_i) \to \partial (\boldR_{>0})^3.$  Hence, $G$ is proper.  

The Jacobian matrix for $F$ at $(s,t,u)$ is 
{\begin{tiny}
\[   J(s,t,u) =
\begin{bmatrix}
\frac{1}{1+s} + \frac{1}{2+s} + \frac{1}{1-s} + \frac{1}{2-s+u}&   0
&  \frac{1}{2-s+u}\\
0 & \frac{2}{1+t} + \frac{2}{1-t-u} & \frac{2}{1-t-u} \\
 \frac{1}{2-s+u} & \frac{2}{1-t-u} 
& \frac{1}{1+u}  +  \frac{1}{2-s+u} + \frac{2}{1-t-u} 
\end{bmatrix}
.\]
\end{tiny}} For all $(s,t,u) \in D,$ the entries of $J(s,t,u)$ are
positive.  For $(s,t,u) \in D,$ the matrix $J(s,t,u)$ is strictly
diagonally dominant, so by the L\'evy-Desplanques Theorem is invertible.
As the Jacobian of $E$ is always invertible, the Jacobian of $F$
is everywhere invertible. 
By Theorem \ref{gordon}, $G$
is a diffeomorphism, hence surjective.   Since $G$ is a
diffeomorphism, $F$ is a diffeomorphism.  

Let $\Sigma(T,\Delta_{\bfa_0})= T \bullet \Delta_{\bfa_0}.$ Theorem
\ref{Fc thm} implies that any algebra in
$\calS_\Lambda(\boldR)$ is isomorphic to precisely one Lie algebra
with structure constants given by an element of the  set
$\calL_7(\boldR) \cap \Sigma(T,\Delta_{\bfa_0}).$
\end{example}

Now we prove Theorem \ref{Fc thm}.
\begin{proof} 
  Let $\alpha \in \calL_\Lambda(\boldR)$ correspond to $\bfa =
  [\alpha_{(i,j}^k]_{(i,j,k) \in \Lambda}$ and let $\beta \in
  \calL_\Lambda(\boldR)$ correspond to $\bfb =
  [\beta_{(i,j}^k]_{(i,j,k) \in \Lambda}.$ Let $\bfs$ denote the sign
  vector for $\alpha$ and let $\bft$ denote the sign vector for
  $\beta.$  We may write $\bfa$ and $\bfb$ as $ \bfa = \bfs \bullet
  |\bfa|$ and $\bfb = \bft \bullet |\bfb|,$ where  $|\bfa| =
  [|\alpha_{(i,j}^k|]_{(i,j,k) \in \Lambda}$ and $|\bfb| =
  [|\beta_{(i,j}^k|]_{(i,j,k) \in \Lambda}.$

Assume that $F_c$ maps $S$ onto $\boldR^d.$  Then $\pi_Y$ maps 
$\Ln(S)$ onto $\boldR^d.$ 
By Lemma \ref{pi lemma},  the set $\Ln(S)$ meets each
orbit of $\rho_Y$ at least once, so 
there exist $\bfa'$ and $\bfb'$ in $S$ so that 
$\Ln \bfa'$ and $\Ln |\bfa|$ are in the same $\rho_Y$ orbit
and 
$\Ln \bfb'$ and $\Ln |\bfb|$ are in the same $\rho_Y$ orbit.  
 Since  $T$ is a simple cross section for the  $\rho_{\hat Y}$ action,
there exists a unique  $\bfs'$ in $T$ so that 
$\bfs$ and $\bfs'$   are in the same $\rho_{\hat Y}$ orbit and 
there exists a unique  $\bft' $ in $T$ so that 
$\bft$ and $\bft'$   are in the same $\rho_{\hat Y}$ orbit.   Recall
that
the action of $D$ on $\calS_\Lambda(\boldR)$ 
is conjugate to the action  $\rho_{\hat Y} \times \rho_{\hat Y}$ on
$\boldZ_2^m \times (\boldR_{>0})^m.$
Therefore $\bfs \bullet |\bfa|$ and 
$\bfs' \bullet \bfa'$ are in the same $D$ orbit, and 
$\bft \bullet |\bfb|$ and $\bft' \bullet
\bfb'$ are in the same $D$ orbit.  
Since isomorphism classes in $\calS_\Lambda(\boldR)$ are $D$ orbits, 
by Theorem \ref{action-isomorphism}, 
the Lie algebras defined by $\bfa = \bfs \bullet |\bfa|$ and 
$\bfs' \bullet \bfa' \in \Sigma(T, S)$
are isomorphic, and the Lie algebras defined by 
$\bfb = \bft \bullet |\bfb|$ and $\bft' \bullet
\bfb' \in \Sigma(T, S)$
are isomorphic. 
We have shown that every element in $\calL_\Lambda(\boldR)$ is
isomorphic to  at least one element of $\Sigma(T,S).$ 

If in addition $F_c$ is a bijection, then $\Ln(S)$ is a simple cross
section for the $\rho_{Y}$ action, and $\bfs' \bullet
\bfa'$ and $\bft' \bullet
\bfb'$ are unique.   By Theorem \ref{action-isomorphism}, 
these points are isomorphic 
if and  only if $\bfs'=\bft'$ and $\bfa'  = \bfb'.$  
  Thus, the   
  Lie brackets 
$\alpha$  and $\beta$ are isomorphic if and
only  if  $\bfs' \bullet \bfa' = \bft' \bullet \bfb'.$
Thus, every element in $\calL_\Lambda(\boldR)$ is
isomorphic to precisely one element of $\Sigma(T,S).$ 
\end{proof}

\section{Aligned pairs of triples, quadruples, and $\Lambda$-subspaces}
\label{combinatorics}

\subsection{Triples and quadruples}
In this section we define some new kinds of objects: aligned pairs of
triples, quadruples of pairs of triples, and the $\Lambda$-subspace
for an index set $\Lambda.$

\begin{defn}\label{(-1)-pair} 
  Let $\bft_1=(i_1,j_1,k_1)$ and $\bft_2=(i_2,j_2,k_2)$ be triples in
  $[n]^3.$  Let $\bfy_{(i_1,j_1,k_1)}$ and
  $\bfy_{(i_2, j_2,k_2)}$ in $\boldR^n$ be the
 corresponding real root vectors  as defined in Section
 \ref{definitions}.
  We say that the triples $\bft_1$ and $\bft_2$ are an {\em
    aligned pair} if the inner product $\la \bfy_{(i_1 j_1,k_1)},  
  \bfy_{(i_2, j_2, k_2)} \ra$ of the corresponding root vectors is $-1.$
\end{defn}
It is possible that an index set $\Lambda$ has no triples that form 
an aligned pair:
\begin{example}\label{l4-b}
  Let $\Lambda \subseteq \Theta_4$ be as in Example \ref{l4}.  There
  are only two distinct triples $(1,2,3)$ and $(1,3,4)$ whose
  corresponding root vectors $\bfy_{(1,2,3)}=(1,1,-1,0)$ and
  $\bfy_{(1,3,4)}=(1,0,1,-1)$ have the matrix product $\bfy_{(1,2,3)}
  \, \bfy_{(1,3,4)}^T$
  equal to zero.  Therefore $\Lambda$ has no aligned pairs of triples.
\end{example}

It is also possible that an index set $\Lambda$ has many triples that form aligned pairs.
\begin{example}\label{example-1b} 
  Let $\Lambda \subseteq \Theta_7$ be as in Example \ref{example-1a}.
Denote the triples in $\Lambda$ by $\bft_1, \ldots,
\bft_s, \ldots, \bft_9,$ where the subscript $s$ ascends concordantly
with the dictionary ordering on $\Lambda:$
\[ \bft_1 = (1,2,3), \bft_2 = (1,3,4), \ldots, \bft_9 = (3,4,7). \]
There are five pairs of triples root vectors have inner product $-1:$
$\bft_4$ and $\bft_6;$ $\bft_2$ and $\bft_7;$ $\bft_5$ and $\bft_7;$
$\bft_3$ and $\bft_8;$ and $\bft_1$ and $\bft_9.$ 
\end{example}

Next we show that every aligned pair of triples $\bft_1$ and $\bft_2$ with
$\bft_1, \bft_2 \in \Theta_n$ 
determines a unique quadruple in $[n]^4,$ and
we assign a
sign to that quadruple.   
Let $\bft_1$ and $\bft_2$ be an  aligned pair of triples in $\Theta_n.$  As
$\bft_1$ and $\bft_2$ are in $\Theta_n,$ their root vectors are of form
$(\ldots, 1, \ldots, 1, \ldots, -1, \ldots)$ where all entries are
zero aside from those indicated.   Because $\bft_1, \bft_2 \in
\Theta_n,$ we know that entries of both triples are distinct and in
ascending order:  $i < j < k$ and either $l < k <m$ or $k < l < m.$
The  product of $\bft_1$ and $\bft_2$ is then
\begin{align*} 
-1 &= (\bfe_i + \bfe_j - \bfe_k )^T(\bfe_k + \bfe_l - \bfe_m) \\
&= \delta_{il} + \delta_{jl}   -1 \quad 
      \text{(since $i,j  < k<m;$  $l < m;$ and $k \ne l$)}. 
\end{align*}
as we have defined the root vectors over $\boldR,$
both $\delta_{il}$ and $\delta_{jl}$ are zero.  
Therefore $i \ne l$ and  $j \ne l.$ It follows that the indices
$i,j$ and $l$ are pairwise distinct.  The index $k$ is characterized
by the fact that is it the unique index occurring in both triples;
therefore $\{i,j,l,m\}$ is the symmetric difference of the sets
$\{i,j,k\}$ and $\{k,l,m\}.$   

Suppose that $\bft_1=(i_1,j_1,k_1)$ and $\bft_2=(i_2,j_2,k_2)$ form an
aligned pair of triples in $\Theta_n.$ Let $\{q_1, q_2, q_3, q_4\}$
denote the symmetric difference of the sets $\{i_1,j_1,k_1\}$ and
$\{i_2,j_2,k_2\},$ where $q_1 < q_2 < q_3 <q_4.$ One may verify that
there exists $r$ so that $\bft_1$ and $\bft_2$ are described by
 one of the six possibilities listed in Table 1.  
Note that Cases 4 and 6 cannot actually occur for a pair of triples in
$\Theta_n$ because of
hypotheses on order relations among the entries of the triples and the
quadruple.   

\begin{table}
\begin{tabular}{c c c}
Case &  Aligned pair of triples & $\sign(\bft_1, \bft
_2)$ \\
\hline 
1 & $\{ \bft_1, \bft_2 \} = \{ (q_1,q_2,r), (r,q_3,q_4)\}$ & 1 \\
2 & $\{ \bft_1, \bft_2 \} = \{ (q_1,q_2,r), (q_3,r,q_4) \}$ & -1\\
3 & $\{ \bft_1, \bft_2 \} = \{ (q_1,q_3,r), (q_2,r,q_4) \}$ & 1\\
4 & $\{ \bft_1, \bft_2 \} = \{ (q_1,q_3,r), (r,q_2,q_4) \}$ & -1 \\
5 & $\{ \bft_1, \bft_2 \} = \{ (q_2,q_3,r), (q_1,r,q_4) \}$ & -1 \\
6 & $\{ \bft_1,\bft_2 \} = \{ (q_2,q_3,r), (r,q_1,q_4) \}$ & 1 \\
 & & \\
\end{tabular}
\caption{Possible aligned pairs of triples and their signs}
\end{table}

Now we are ready to define the quadruple associated to a pair of
triples, and the sign associated to a pair of quadruples.
\begin{defn}\label{quad-sign}
  Let $\bft_1=(i_1,j_1,k_1)$ and $\bft_2=(i_2,j_2,k_2)$ be an
  aligned pair of triples in $\Theta_n.$ The {\em quadruple
    $q(\bft_1,\bft_2)$ for $\bft_1$ and $\bft_2$} is defined to be
  \[ q(\bft_1,\bft_2) = (q_1,q_2,q_3,q_4)\in [n]^4,\] where
  $\{q_1,q_2,q_3,q_4\}$ is the symmetric difference of the sets
  $\{i_1,j_1,k_1\}$ and $\{i_2,j_2,k_2\}$ and $q_1 < q_2 < q_3 < q_4.$

Define the sign $\sign(\bft_1,\bft_2) $ of $\bft_1$ and $\bft_2$ by
defining $\sign(\bft_1,\bft_2) $ to be $-1$ in Cases 2, 4 and 5 of Table
1, and defining $\sign(\bft_1,\bft_2) $ to be $1$ in Cases 1, 3 and 6.

We say that the triple $\bfs$ is a {\em common triple} for 
quadruples $\bfq_1$ and $\bfq_2$ if there are triples $\bft_1, \bft_2$
so that  $\bfq_1 = q(\bfs,\bft_1)$ and $\bfq_2 = q(\bfs,\bft_2).$

Let $Q$ be the set of quadruples associated to any aligned pairs of
triples in $\Lambda.$ We say that a quadruple in $Q$ has multiplicity
$m$ if it arises from exactly $m$ distinct pairs of triples.
\end{defn}

Note that  $q(\bft_1,\bft_2) = q(\bft_2,\bft_1) .$   

\begin{example}\label{example-1c} 
  Let $\Lambda \subseteq \Theta_7$ be as in Example \ref{example-1a}
  and \ref{example-1b}.
The quadruples
associated to the aligned pairs are
\begin{gather*}
q(\bft_4,\bft_6) = q(\bft_2,\bft_7) = (1,2,3,6) \\
q(\bft_5,\bft_7) = q(\bft_3,\bft_8) = q(\bft_1,\bft_9) = (1,2,4,7). 
\end{gather*}
The set of quadruples for $\Lambda$ is
  \[ Q = \{ (1,2,3,6), (1,2,4,7) \} .\] 
The quadruple $(1,2,3,6)$ has multiplicity two and the quadruple
 $(1,2,4,7)$ has multiplicity  three. 
The signs of the pairs are 
\begin{gather*} 
\sign(\bft_4,\bft_6) = \sign(\bft_5,\bft_7) = -1, \quad \text{and} \\
  \sign(\bft_2,\bft_7)
= \sign(\bft_3,\bft_8) = \sign(\bft_1,\bft_9) = 1.\end{gather*}
\end{example}

If there are no aligned pairs of triples in $\Lambda,$ then there are
no quadruples associated to $\Lambda.$ This is the case with Examples
\ref{l4} and
\ref{l4-b}.  It is also possible that there is only one quadruple
associated to an index set.

\begin{example}\label{1qm3-b} 
  Enumerate $\Lambda$ from Example \ref{1qm3-a} in dictionary
  order.  There are three aligned pairs: $\bft_4$ and $\bft_5,$
  $\bft_2$ and $\bft_6,$ and $\bft_1$ and $\bft_7.$ For each of these
  pairs, the associated quadruple is $(1,2,3,7).$ That quadruple has
  multiplicity three.
\end{example}

\subsection{A subspace determined by the index set}
If two aligned pairs in an index set $\Lambda$ have the same quadruple
associated to them, they determine a vector $\bfw$ defined as
follows. In Theorem \ref{-1}, it will be seen that all such vectors
are contained in the left null space of the root matrix $Y$ for
$\Lambda.$
\begin{defn}\label{w-defn}
Fix an index set 
$\Lambda \subseteq \Theta_n$ and enumerate its elements
so that $\Lambda 
 = \{ \bft_1, \ldots, \bft_m\}.$
Suppose that 
$\{\bft_{m_1}, \bft_{m_2}\}$  and   $\{\bft_{m_3}, \bft_{m_4}\}$
  are two different aligned pairs from $\Lambda$ with the same quadruple
\[  q(\bft_{m_1},\bft_{m_2}) = q(\bft_{m_3},\bft_{m_4}).\]

Define the 
$m \times 1$ column 
vector $\bfw$ in $K^m$ {\em associated to the aligned pairs $\{
  \bft_{m_1}, \bft_{m_2} \} $ and $\{ \bft_{m_3}, \bft_{m_4}\} $} to be
\[ \bfw(m_1,m_2,m_3,m_4) = \bfe_{m_1} +\bfe_{m_2} - \bfe_{m_3} -
\bfe_{m_4}. \] 
If we omit mention of the field we assume that the field is
$\boldR.$ 

Define the subspace $W_{\Lambda}(K)$ of $K^m$ to be the span of all
vectors $\bfw(m_1, m_2, m_3, m_4) $ over $K$ arising from
all aligned pairs $\{\bft_{m_1},\bft_{m_2}\}$ and
$\{\bft_{m_3},\bft_{m_4}\}$ sharing the same quadruple:
\[W_{\Lambda}(K) = \myspan_{K} \{ \bfw(m_1, m_2, m_3, m_4) \, : \,
q(\bft_{m_1}, \bft_{m_2})=q(\bft_{m_3},\bft_{m_4}) \}.\] We call
$W_\Lambda(K)$ the {\em $\Lambda$-subspace} of
$K^m.$  
\end{defn}

Note that if $\{ \bft_{m_1}, \bft_{m_2}\} $ and $\{ \bft_{m_3},
\bft_{m_4}\} $ are aligned pairs with the same quadruple,
\[  \bfw(m_1,m_2,m_4,m_3) 
= \bfw(m_1,m_2,m_3,m_4) = \bfw(m_2,m_1,m_3,m_4)\]
and
\[ \bfw(m_1,m_2,m_3,m_4) = - \bfw(m_3,m_4,m_1,m_2).\]

\begin{example}\label{example-1d}  
  Let $\Lambda$ be as in Examples \ref{example-1a}, \ref{example-1b}, 
  and \ref{example-1c}.  The aligned pairs $\{ \bft_4 , \bft_6 \}$ and
  $\{ \bft_2, \bft_7 \}$ have associated quadruple $(1,2,3,6).$ The
  vector in $\boldR^9$
  associated to these two aligned pairs is
\[ \bfw(4,6,2,7) = \bfe_4 + \bfe_6 - \bfe_2 - \bfe_7.\]

We also have (up to sign changes) three other vectors 
$\bfw(m_1,m_2,m_3,m_4)$ arising from the three other aligned pairs of triples, giving 
a total  of four vectors (up to signs) 
\begin{align*}
\bfw_1 & = \bfw(4,6,2,7)= (0,-1,0,1,   0,1,-1,0,   0)^T \\
\bfw_2 &= \bfw(5,7,3,8) = (0,0,-1,0,    1,0,1,-1,  0)^T \\
\bfw_3 &= \bfw(3,8,1,9)= (-1,0,1,0,0,0,0,1,-1)^T \\
\bfw_4 &= \bfw(1,9,5,7) = (1,0,0,0,-1,0,-1,0,1)^T .
\end{align*}
 Observe that $\bfw_1 + \bfw_2 + \bfw_3 + \bfw_4 = \bfzero$ in
$\boldR^9.$ 

The real  $\Lambda$-subspace is the three-dimensional subspace 
\[W_\Lambda(\boldR) = \myspan_\boldR \{ \bfw_1, \bfw_2, \bfw_3 \}
\subseteq \boldR^9. \]
\end{example}

The   $\Lambda$-subspace
 is always a subspace of the left null space of the root  matrix
$Y(K)$ associated to $\Lambda.$ 
\begin{thm}\label{-1}
  Let $\Lambda \subseteq \Theta_n,$ and let $W_\Lambda(K)$ be the
  $\Lambda$-subspace of $K^m$ as in Definition \ref{w-defn}.  Let
  $Y(K)$ be the $K$ root matrix for $\Lambda.$ Then $W_\Lambda(K)$ is
  a subspace of $\Null(Y(K)^T).$
\end{thm}

\begin{proof} 
Refer to Table 1.  If the pair of triples $\{ \bft_{m_1}, \bft_{m_2}
\}$ has associated quadruple $(q_1, q_2, q_3, q_4),$ then in all six
cases
\[ \bfy_{\bft_{m_1}} + \bfy_{\bft_{m_2}} = \bfe_{q_1} + \bfe_{q_2} +
\bfe_{q_3} - \bfe_{q_4}.\] Therefore, if $\{ \bft_{m_1}, \bft_{m_2}
\}$ and $\{ \bft_{m_3}, \bft_{m_4}\}$ are two aligned pairs of triples
both having the same quadruple $(q_1, q_2, q_3, q_4),$
\[\bfy_{\bft_{m_1}} + \bfy_{\bft_{m_2}} = \bfy_{\bft_{m_3}} +
\bfy_{\bft_{m_4}}.\]

Recall that the vector $\bfy_{m_s} $ is the $m_s$th row of the root
matrix $Y(K).$ Hence the dependency
\begin{align}
 \bfy_{m_1} + \bfy_{m_2} - \bfy_{m_3} - \bfy_{m_4} = \bfzero 
\end{align}
of rows may be written as $\bfw Y(K) = 0,$ where
\[ \bfw = \bfw(m_1,m_2,m_3,m_4) = \bfe_{m_1} + \bfe_{m_2} - \bfe_{m_3} - \bfe_{m_4}.\]  
 Therefore $Y^T \bfw^T = \bfzero$ and $\bfw^T$ is in the null space of $Y.$
\end{proof}

\begin{defn} 
The subset $\Lambda$ of $\Theta_n$ is said to be {\em
null space spanning} over $K$ if the $\Lambda$-subspace 
of $K^n$  is equal to the
full null space: $W_\Lambda(K) = \Null(Y(K)^T).$
\end{defn}

We revisit Example \ref{1qm3-a}.
\begin{example}\label{1qm3-c} Let $\Lambda$ be as in Examples
\ref{1qm3-a} and \ref{1qm3-b}.  We saw that the vectors $\bfw_1 =
(0,1,0,-1,-1,1,0)^T $ and $\bfw_2 = (1,0,0,-1,-1,0,1)$ spanned
$\Null(Y^T).$ But these vectors are just the vectors determined by the
aligned pairs of triples we saw in Example \ref{1qm3-b}: $\bfw_1 =
\bfw(2,6,4,5)$ and $\bfw_2 = \bfw(1,7,4,5).$ Hence $\Lambda$ is
null space spanning.
\end{example} 

We leave it to the reader to verify that
the index sets $\Lambda$ in Example \ref{1qm2-a}
and  Example \ref{example-1a} are null space spanning.  Not all
index sets are null space spanning, as the following example from
dimension eight shows.

\begin{example}\label{not null spanning} For the subset
$\Lambda$ of $\Theta_8$ defined by
\begin{multline*} \Lambda = \{ (1,2,4), (1,3,5), (1,4,6), (1,5,7),
(1,7,8),\\ (2,3,6), (2,4,7), (2,6,8), (3,5,8) \},\end{multline*} the
only quadruple is $(1,2,4,8).$ It has multiplicity two, arising
from the two aligned pairs of triples $\{ \bft_5,\bft_7\} = \{(1,7,8),(2,4,7)\}$ and
$\{\bft_3,\bft_8\}
= \{(1,4,6),(2,6,8)\}.$ However, the span of $\bfw_1 = \bfe_5 + \bfe_7
- \bfe_3 -\bfe_8$ is not the full left null space of the associated
root matrix $Y.$ Actually, $\Null(Y^T)$ is spanned by $\bfw_1$ and
$\bfw_2 = (1,0,1,-1,-1,-1,0,0,1)^T.$
\end{example}

\begin{remark} 
Even in the case that $\Lambda$ is not
null space spanning, the vectors $\bfw_i$ may still be used
as part of a basis for the tangent space.  As they have all entries
of zero except for two 1's and two -1's, a basis including them may 
be simpler than a full basis found by a computer algebra system.  
\end{remark}

\subsection{Criterion for injectivity}
Recall that in Example \ref{example-1a}, the Jacobian matrix of the
mapping $F$ was nonsingular because it was diagonally dominant.  For
a general index set $\Lambda,$ one can extract a condition on the combinatorics
of the set of quadruples that makes the Jacobian matrix of the mapping
$F$ diagonally dominant.

\begin{lemma}\label{main lemma} 
  Let $\Lambda\subseteq \Theta_n$ be an index set of cardinality $m$
  with associated root matrix $Y.$ Suppose that $\Lambda$ is null
  space spanning and that $\calB = \{ \bfw_1, \bfw_2, \ldots,
  \bfw_d\}$ is a basis for $\Null(Y^T),$ where for $i=1, \ldots, d,$
  the vector $\bfw_i$ arises from the quadruple $q_i=
  (t_1^i,t_2^i,t_3^i,t_4^i) \subseteq [n]^3:$
\[ \bfw_i = \bfw(t_1^i,t_2^i,t_3^i,t_4^i) = \bfe_{t_1^i} +\bfe_{t_2^i}
-\bfe_{t_3^i} -\bfe_{t_4^i}.\]
Suppose that the set of quadruples for $\Lambda$ 
satisfies the following conditions.
\begin{enumerate}
\item{For each quadruple $q_i,$ there is at least one element from
the list $t_1^i,t_2^i,t_3^i,t_4^i$ that does not occur in any of the other
quadruples. }\label{quadhyp1}
\item{For each quadruple $q_i,$ each element $t_1^i,t_2^i,t_3^i,t_4^i$
of the quadruple occurs at most once in all of the other quadruples.}\label{quadhyp2}
\end{enumerate}
Let $\bfa(0) \in (\boldR_{>0})^m.$ For $\bfs=(s_1, \ldots, s_m) \in
\boldR^m,$ let $\bfa(\bfs) = \bfa(0) + \sum_{i = 1}^d s_i \bfw_i .$
and let  $\Delta_{\bfa_0} = \{ \bfa(\bfs) \, : \, \bfs \in \boldR^m\} \cap
(\boldR_{>0})^m.$ Then $\Ln(\Delta_{\bfa_0})$ is a simple cross section for the $\rho_Y$ action.
\end{lemma}

\begin{proof} 
Denote the coordinate functions of $\bfa(\bfs)$ by $a_1, \ldots, a_d.$
Let $D$ denote the set of values for $\bfs$ which parametrize
$\Delta_{\bfa_0}:$
\[ D =  \{ \bfs \, :  \, a_i(\bfs) > 0 \enskip \text{for all} \enskip
i=1, \ldots, d \} .\]
For $i=1,
\ldots, d,$ the function $a_i$ is given by 
\[ a_i(\bfs) = a_i(0) + \sum_{i \in \{ t_1^k, t_2^k\}} s_k - \sum_{i
\in \{ t_3^k, t_4^k\}} s_k ,\] 
so  for $\bfs \in D,$
\[\ln ( a_i(\bfs)) = \ln \left(a_i(0) + \sum_{i \in \{ t_1^k, t_2^k\}} s_k -
\sum_{i \in \{ t_3^k, t_4^k\}} s_k \right). \] 
The partial derivatives of the functions $a_i$ with respect to $s_j$ are 
\[ \frac{\partial a_i}{\partial s_j} =
\begin{cases} 
0 & i \not \in \{ t_1^j, t_2^j, t_3^j, t_4^j\} \\ 
1 & i  \in \{ t_1^j, t_2^j\} \\ 
-1 & i  \in \{ t_3^j, t_4^j\}.
\end{cases}\] Fix $i.$ The $i$th coordinate function of $F = \pi_Y
\circ \Ln \circ \bfa $ is
\begin{align*}
 F_i(\bfs) &= (\pi_Y \circ \Ln \circ \bfa)_i (\bfs)  \\
&= \bfw_i \cdot  (\pi_Y \circ\Ln \circ \bfa) (\bfs)\\
&= \ln(a_{t_1^i} (\bfs) )
+\ln(a_{t_2^i} (\bfs) )- \ln(a_{t_3^i} (\bfs) ) -\ln(a_{t_4^i} (\bfs)
).\end{align*}
Then 
\[ \frac{\partial F_i}{\partial s_i} = \sum_{k \in
\{t_1^i,t_2^i,t_3^i,t_4^i\}} 
\frac{\partial F_i}{\partial a_k}\frac{\partial a_k}{\partial s_i}=  
\sum_{k \in
\{t_1^i,t_2^i,t_3^i,t_4^i\}} \frac{1}{a_{k}(\bfs)}, \] 
while for $i \ne j,$
\[ \frac{\partial F_i}{\partial s_j} = \sum_{k \in
  \{t_1^i,t_2^i,t_3^i,t_4^i\}} \frac{\partial F_i}{\partial
  a_k}\frac{\partial a_k}{\partial s_j} = \sum_{k \in
  \{t_1^i,t_2^i,t_3^i,t_4^i\} \cap \{t_1^j,t_2^j,t_3^j,t_4^j\} }
\frac{1}{a_{k}(\bfs)} . \] Therefore the sum of the nondiagonal
entries in row $i$ of the Jacobian matrix for $F$ is
\[ \sum_{i \ne j }  \frac{\partial F_i}{\partial s_j}=  \sum_{i \ne j }\sum_{k \in \{t_1^i,t_2^i,t_3^i,t_4^i\} \cap
\{t_1^j,t_2^j,t_3^j,t_4^j\} } \frac{1}{a_{k}(\bfs)} .\] 

If it is $t_3^i$ or $t_4^i$ which satisfies hypotheses
\eqref{quadhyp1} of the lemma, we can replace $\bfw_i=
\bfw(t_1^i,t_2^i,t_3^i,t_4^i)$ by $-\bfw_i=
\bfw(t_3^i,t_4^i,t_1^i,t_2^i)$ in the parametrization without changing
whether or not the Jacobian is nonzero.  Then since
$\bfw(t_1^i,t_2^i,t_3^i,t_4^i)=\bfw(t_2^i,t_1^i,t_3^i,t_4^i)$ we may
assume without loss of generality that $t_1^i \not \in
\{t_1^j,t_2^j,t_3^j,t_4^j\}$ for all $j \ne i.$

By hypothesis \eqref{quadhyp2}, 
\[ \sum_{i \ne j }  \frac{\partial F_i}{\partial s_j} \le  
\sum_{k \in 
\{t_2^1,t_3^1,t_4^1\} } \frac{1}{a_{k}(\bfs)} <\frac{\partial
F_i}{\partial s_i} 
.\] 
Thus, the Jacobian is diagonally dominant.  Therefore by the
L\'evy-Desplanques Theorem, it is nonsingular.  
\end{proof}

\section{The Jacobi Identity}\label{jacobi identity}

The following theorem from \cite{payne-09b} reformulates
the  Jacobi Identity in terms of structure constants in
the case that the basis is triangular.  We have rephrased the
hypotheses of the theorem using the language of aligned pairs and
associated quadruples.  

\begin{thm}[Theorem 7, \cite{payne-09b}] \label{jacobicondition} 
Let
$\calB = \{x_i\}_{i=1}^n$ be a triangular basis for\/ $\boldR^n.$
Let the vector\/ $[\alpha_{ij}^k]_{(i,j,k) \in \Lambda}$ 
of nonzero structure constants indexed by\/ $\Lambda
\subseteq \Theta_n,$ together with the basis $\calB,$ define a
skew-symmetric product $\mu: \boldR^n \times \boldR^n \to \boldR^n.$

The product $\mu$ defines a Lie algebra if and only if, whenever there
exists an aligned pair of triples $\bft_p$ and $\bft_r$ in $\Lambda$
with associated quadruple $(i,j,k,m)$ then the equality
\begin{equation}\label{jacobi-constraint} \sum_{s < m} \alpha_{ij}^s
\alpha_{sk}^m + \alpha_{jk}^s \alpha_{si}^m + \alpha_{ki}^s
\alpha_{sj}^m = 0
\end{equation} holds.  Furthermore, a term of form $\alpha_{ij}^l
\alpha_{lk}^m$ in Equation \eqref{jacobi-constraint} is nonzero if and
only if the corresponding triples $ (i,j,l)$ or $(j,i,l),$ and $
(l,k,m)$ or $(k,l,m),$ are an aligned pair.
\end{thm}
Although Theorem \ref{jacobicondition} is proved over $\boldR$ in 
\cite{payne-09b}, it is remains true over arbitrary fields.
Now we prove Theorem \ref{ji-quads-signs}.

\begin{proof}

  We need to see how to rewrite Equation \eqref{jacobi-constraint} of
Theorem \ref{jacobicondition}, in which the subscripts $i,j$ and $k$
are ordered cyclically in each summand, to a convention in which all
the terms in the sum are of the form $\alpha_{ij}^k \alpha_{kl}^m$
with $i < j<k$ and $k<l<m.$

  Suppose that $(q_1,q_2,q_3,q_4)$ is in the set of quadruples for
$\Lambda.$ Then $q_1 < q_2 < q_3 < q_4.$ Suppose that triples $\bft_1$
and $\bft_2$ are an aligned pair with quadruple $(q_1,q_2,q_3,q_4).$
There is an $r \in [n]$ so that one of the triples is in the set $\{
(q_1,q_2,r), (q_1,q_3,r), (q_2,q_3,r)\}.$ 

 We need to see whether the corresponding nonvanishing term of
Equation \eqref{jacobi-constraint} changes sign if we re-order the
lower subscripts so they are all in ascending order.  There are six
cases, as listed in Table 2, to consider.  In each case, the
left side of the equality in the third column 
expresses term from the Jacobi Identity so
that the subscripts come from the triples as in Equation
\eqref{ji-qt}, and on the right side, subscripts are in a cyclic form
of $(q_1, q_2, q_3)$ as in Equation \eqref{jacobi-constraint}.
In each case, the sign change is determined by the value of 
$\sign(\bft_1,\bft_2)$
given in Table 1.  By making appropriate substitutions, we get
Equation \eqref{ji-qt}.
\begin{table}
\begin{tabular}{c c c} Case & Aligned pair of triples & Product \\
\hline 
1 & $\{ \bft_1, \bft_2 \} = \{ (q_1,q_2,r), (r,q_3,q_4)\} $ &
$\alpha_{q_1 q_2}^r \alpha_{r q_3}^{q_4} = \alpha_{q_1 q_2}^{r}
\alpha_{r q_3}^{q_4} $\\ 
2 & $\{ \bft_1, \bft_2 \} = \{ (q_1,q_2,r),
(q_3,r,q_4) \}$ & $\alpha_{q_1 q_2}^r \alpha_{q_3 r}^{q_4} = -
\alpha_{q_1 q_2}^{r} \alpha_{r q_3}^{q_4} $\\ 
3 &$\{ \bft_1, \bft_2 \}
= \{ (q_1,q_3,r), (q_2,r,q_4) \} $ & $\alpha_{q_1 q_3}^r \alpha_{q_2
r}^{q_4} = \alpha_{q_3 q_1}^{r} \alpha_{r q_2}^{q_4}$\\ 
4& $\{ \bft_1,
\bft_2 \} =\{ (q_1,q_3,r), (r,q_2,q_4) \} $ & $\alpha_{q_1 q_3}^r
\alpha_{r q_2}^{q_4} = - \alpha_{q_3 q_1}^{r} \alpha_{r q_2}^{q_4} $\\
5& $\{ \bft_1, \bft_2 \} = \{ (q_2,q_3,r), (q_1,r,q_4) \}$ &
$\alpha_{q_2 q_3}^r \alpha_{q_1 r}^{q_4} = - \alpha_{q_2 q_3}^{r}
\alpha_{r q_1}^{q_4} $\\ 
6& $\{ \bft_1, \bft_2 \} = \{ (q_2,q_3,r),
(r,q_1,q_4) \}$ & $\alpha_{q_2 q_3}^r \alpha_{r q_1}^{q_4} =
\alpha_{q_2 q_3}^{r} \alpha_{r q_1}^{q_4} $\\ &&\\
\end{tabular}
\caption{Sign changes in the Jacobi Identity}
\end{table}

\end{proof}

Corollaries \ref{JI automatic} and \ref{obstruction} follow
immediately from Theorem \ref{ji-quads-signs}.  In the first case,
when there are no quadruples, there are no constraints from the Jacobi
Identity.  In the second case, the Jacobi Identity becomes one
equation of form $ \alpha_{i_p j_p}^{k_p} \alpha_{i_r j_r}^{k_r} =0,$
which has no nonzero solutions in the field $K.$

We illustrate Theorem \ref{ji-quads-signs} with some examples.

\begin{example}\label{l4-c}
Let $\Lambda$ be as in Examples \ref{l4} and \ref{l4-b}.  Since there
are no quadruples for $\Lambda,$ the Jacobi Identity automatically
holds for all elements of $\calS_\Lambda(K)$ by Corollary \ref{JI
automatic}.  We may now conclude the algebra in
$\calS_\Lambda(\boldR)$ with $\alpha_{12}^3 = 1$ and $\alpha_{13}^4 =
1$ is a Lie algebra, and all Lie algebras in $\calS_\Lambda(\boldR)$
are isomorphic to it.
\end{example}

\begin{example}\label{1qm2-b} Let $\Lambda$ be as in Example
\ref{1qm2-a}.  We have two pairs of aligned triples, $\{ \bft_1,
\bft_6\} = \{(1,2,4),(3,4,7)\}$ and $\{ \bft_2,\bft_5\} = \{ (1,3,5),
(2,5,7) \}$ both having quadruple $(1,2,3,7).$ The signs of the pairs
are $\sign(\bft_1,\bft_6) =-1$ and $\sign(\bft_2,\bft_5) =1.$ By
Theorem \ref{ji-quads-signs}, the Jacobi Identity for
$\calS_\Lambda(\boldR)$ is equivalent to the equation
\[ - \alpha_{12}^4 \alpha_{34}^7 + \alpha_{13}^5 \alpha_{25}^7 = 0. \]
In Example \ref{1qm2-a}, we found that the set $\Sigma(T,\Delta)$ in
\eqref{1qm2-a rep set} was a parametrizing set for
$\calS_\Lambda(\boldR).$ To determine which algebras among these
satisfy the Jacobi Identity, we may solve
\[ |\alpha_{12}^{4} \alpha_{34}^7| = |\alpha_{13}^5
\alpha_{25}^7|, \quad \sgn(\alpha_{12}^{4} \alpha_{34}^7 ) =
\sgn(\alpha_{13}^5 \alpha_{25}^7) \] for elements of
$\Sigma(T,\Delta).$ Substituting the expression from Equation
\eqref{1qm2-a rep set} for an element of $\Sigma(T,\Delta)$ into the
first equation gives
$1+s = 1-s.$ Hence, $s=0.$ The sign vector must be
$(0,0,0,0,0,0)$ in order for the second equation to be true.  We
conclude that every Lie algebra in $\calS_\Lambda(\boldR)$ is
isomorphic to the one with structure constants $\alpha_{ij}^k = 1$ for
all $(i,j,k) \in \Lambda.$
\end{example}

The next example is of the type described by Part \ref{1qm3} of Theorem \ref{analyze-exhaustive}. 
\begin{example}\label{2qm2b-a} 
Let
\[ \Lambda= \{(1,3,4), (1,4,6), (1,6,7), (1,7,8), (2,3,6), (2,4,7),
(2,6,8), (3,5,8) \}.\]
There are four aligned pairs. The pairs 
\[ \{\bft_1, \bft_6 \} = \{(1,3,4), (2,4,7)\} \enskip \text{and}  
\enskip \{\bft_3, \bft_5 \} = \{
(1,6,7), (2,3,6) \}\] have the quadruple $(1,2,3,7)$ associated to
them, while 
\[ \{\bft_2,\bft_7\} = \{(1,4,6), (2,6,8)\} \enskip  \text{and} \enskip  
\{\bft_4,\bft_6 \} = \{ (1,7,8), (2,4,7)\}\]have
the quadruple $(1,2,4,8)$ associated to them.  Accordingly, the
vectors
\begin{align*} \bfw_1 &= (1,0,-1,0,-1,1,0,0)^T , \text{and} \\ \bfw_2
&= (0,1,0,-1, 0,-1,1,0)^T
\end{align*} are in $\Null(Y^T)$ as guaranteed by  Proposition
\ref{-1}.  The null space of $Y^T$ is two-dimensional so $\Lambda$ is
null space spanning.

For $s, t \in \boldR,$ let
\[ \bfa(s,t) = (1,1,1,1,1,1,1,1)^T+s\bfw_1 + t\bfw_2.
\] The simple cross section for the $\rho_{\hat Y}$ action of
$\boldZ_2^n$ on $\boldZ_2^m$ has only one element, $\bfzero \in \boldZ_2^8.$
Methods from Section \ref{cross sections} may be used to show that
every Lie algebra in $\calS_\Lambda(\boldR)$ is isomorphic to exactly
one Lie algebra in
\begin{equation}\label{soln-2} \Sigma(T,\Delta) = \{
(1+s,1+t,1-s,1-t,1-s,1+s-t,1+t,1) \, : \, (s,t) \in D \}
,\end{equation} where
\[ D = \{ (s,t) \, : \, |s| < 1, |t| < 1, t < 1 + s \}. \] 
  By Theorem \ref{jacobicondition}, the Jacobi Identity for elements of products
$\mu$ in $\calS_\Lambda(\boldR)$ is equivalent to the system of equations
\[
 \alpha_{13}^4 \alpha_{24}^7 - \alpha_{16}^7
\alpha_{23}^6 = 0,  \quad \alpha_{14}^6
\alpha_{26}^8 - \alpha_{17}^8 \alpha_{24}^7 = 0. \]
Equivalently,  the absolute values of the structure
constants satisfy 
\begin{equation}\label{mag-constraint}
|\alpha_{13}^4 \alpha_{24}^7| = |\alpha_{16}^7 \alpha_{23}^6|
\quad \text{and} \quad 
 |\alpha_{14}^6 \alpha_{26}^8| = |\alpha_{17}^8
\alpha_{24}^7|\end{equation}  while simultaneously, their signs
 satisfy
\begin{equation}\label{signs are} \sgn(\alpha_{13}^4 \alpha_{24}^7)=
\sgn(\alpha_{16}^7 \alpha_{23}^6), \enskip \text{and} \enskip
\sgn(\alpha_{14}^6 \alpha_{26}^8) = \sgn(\alpha_{17}^8 \alpha_{24}^7).
\end{equation}
Since the sign vector $\bfzero \in T$ has all zero entries, all signs
are positive for points in the simple cross section, and the
equalities in \eqref{signs are} in hold.  Substituting the values of
$\alpha_{ij}^k, (i,j,k) \in \Lambda,$ from \eqref{soln-2} into Equation 
\eqref{mag-constraint} yields
\begin{align*} (1+s)(1+s - t) &= (1-s)^2 \\ (1+t)^2&=
(1-t)(1+s-t) .\end{align*} It is not hard to show that the only
solution yielding positive values for the squares of the structure
constants is $s=t=0.$  Substituting $s=t=0$ into Equation
\eqref{soln-2} gives
\[  \Sigma(T,\Delta) \cap \calL_\Lambda(\boldR) = \{ (1,1,1,1,1,1,1,1)\}.\]
Thus, every Lie algebra in $\calL_\Lambda(\boldR)$ 
is isomorphic to the one with the structure
constants
\[ [\alpha_{ij}^k]_{(i,j,k) \in \Lambda} = (1,1,1,1,1,1,1,1).\]
\end{example}

\begin{example}\label{1qm3-d} Let $\Lambda$ be as in Examples
\ref{1qm3-a}, \ref{1qm3-b}, and \ref{1qm3-c}. 
 We have established  that each Lie   algebra in $\calS_\Lambda(\boldR)$ is represented
exactly once in the set of points 
\[\Sigma(T, \Delta_{\bfa_0}) = \{ \bfa(s,t) \, : \, s,t \in
D\}, \] 
where the vector $\bfa(s,t) = [\alpha_{ij}^k]_{(i,j,k) \in
  \Lambda}$ is equal to  
 \[ \left(1+t,2+s,
1,1-s-t, 1-s-t, \pm(2+s), \pm( 1+t) \right)^T  \]  and 
\[ D = \{ (s,t) \, : \, s > -2, t > -1, \text{and} \, s+t < 1 \}
\subseteq \boldR^2.  \]
The set $\Sigma(T, \Delta_{\bfa_0}^{1/2})$ of points of form
\[
 \left((1+t)^{1/2},(2+s)^{1/2}, 1,(1-s-t)^{1/2}, (1-s-t)^{1/2},
   \pm(2+s)^{1/2}, \pm( 1+t)^{1/2} \right)^T  ,\]
with $(s,t) \in D,$
is a second parametrizing set.  

 By Theorem \ref{ji-quads-signs}, the Jacobi
 Identity for $\calS_\Lambda(\boldR)$ 
is equivalent to
\[ \alpha_{13}^5 \alpha_{25}^7 - \alpha_{23}^6 \alpha_{16}^7 -
\alpha_{12}^4 \alpha_{34}^7 = 0 . \] 
Note that if the signs are given by the sign vector $\bfe_6 \in
\boldZ_2^7,$  all terms on the left side of the equation are negative,
so there are no solutions.  Hence the sign vector for a Lie algebra in
$\Sigma(T, \Delta_{\bfa_0})$ or $\Sigma(T, \Delta_{\bfa_0}^{1/2})$ must be in 
\[  T_1 = \{(0,0,0,0,0,0,0),(0,0,0,0,0,0,1),(0,0,0,0,0,1,1)\}. \]

First we consider the Jacobi Identity for products in 
$\Sigma(T, \Delta_{\bfa_0}^{1/2}):$ 
\begin{multline}\label{JI 1qm3-d} -\sign(\alpha_{13}^5) \sign
(\alpha_{25}^7) (2 + s) +\\ \sign(\alpha_{23}^6) \sign (\alpha_{16}^7)
(1-s-t) + \sign(\alpha_{12}^4)\sign (\alpha_{34}^7) (1 + t) =
0.\end{multline}

There are now three cases to consider, one for each component of 
$\Sigma(T_1,\Delta^{1/2}).$
\begin{itemize}

\item{If sign vector is $(0,0,0,0,0,0,0),$ then all signs are positive,
and Equation \eqref{JI 1qm3-d} becomes
\[ 0=-(2+s) +(1-s-t) + (1+t) = -2s, \] which has solutions
$s=0, -1 < t < 1$ in $D.$ Therefore, structure constants are in
\begin{tiny}
\[ \Sigma_1 = \{ ( \sqrt{1+t}, \sqrt{2}, 1, \sqrt{1-t}, \sqrt{1-t},
\sqrt{2}, \sqrt{1+t} ) \, : \, -1 < t < 1 \}. \]
\end{tiny}
}
\item{ If the sign vector is $(0,0,0,0,0,0,1),$ we get
\[ 0 = -(2+s) + (1-s-t) - (1 +t) = -2 - 2s - 2t,\] hence $s + t=-1. $
and structure constants are
\begin{tiny}
\[ \Sigma_3 = \{ (\sqrt{-s}, \sqrt{2+s}, 1, \sqrt{2}, \sqrt{2},
\sqrt{2+s}, -\sqrt{-s}) \, : \, -2 < s < 0 \}\] \end{tiny}}
\item{In the last case, when signs are encoded by $(0,0,0,0,0,1,1),$
we have
\[ 0 = (2+s) + (1-s-t) - (1 +t) = 2  - 2t, \] hence $ t=1. $
The
corresponding structure constants are
\begin{tiny}
\[ \Sigma_2 = \{ ( \sqrt{2}, \sqrt{2+s}, 1, \sqrt{-s},
\sqrt{-s}, -\sqrt{2+s}, -\sqrt{2}) \, : \, -2 < s < 0 \}. \]
\end{tiny} }\end{itemize} In sum, each Lie algebra in $\calS_\Lambda(\boldR)$
is isomorphic to precisely one Lie algebra whose structure constants
are encoded by a vector in $\Sigma = \Sigma_1 \cup \Sigma_2 \cup
\Sigma_3.$

Now we find an alternate parametrization using 
$\Sigma(T,\Delta)$ instead of $\Sigma(T,\Delta^{1/2}).$ In this case, the Jacobi
Identity becomes
\begin{multline}\label{JI} -\sign(\alpha_{13}^5) \sign (\alpha_{25}^7)
(2 + s)^2 +\\ \sign(\alpha_{23}^6) \sign (\alpha_{16}^7) (1-s-t)^2 +
\sign(\alpha_{12}^4)\sign (\alpha_{34}^7) (1 + t)^2 = 0.\end{multline}
When the sign vector is $(0,0,0,0,0,0,0),$ the Jacobi Identity becomes
\[ 2t^2 + 2st - 6s - 2 = 0,\] so
\begin{tiny}
\[ \Sigma_1^\prime =
\{ ( 1 + t, 2+ s(t), 1,1-s(t)-t, 1-s(t)-t, 2 + s(t),1+t) 
\, : \, -1 < t < 1 \}.\]
\end{tiny} 
where $s(t) = \frac{1-t^2}{t - 3}.$ We may solve for
$\Sigma_2^\prime$ and $\Sigma_3^\prime$ in a similar manner to get the
parametrizing set $\Sigma^\prime = \Sigma_1^\prime \cup
\Sigma_2^\prime \cup \Sigma_3^\prime,$ where
\begin{tiny}
\begin{align*} 
\Sigma_2^\prime
&= \{ ( 1 + t_2(s), 2+ s, 1,1-s-t_2(s), 1-s-t_2(s), 2 + s, -1-t_2(s)) \, : \,
-2<s<0 \} \\
\Sigma_3^\prime &= \{ ( 1 + t_3(s), 2+ s, 1, 1-s-t_3(s),
1-s-t_3(s), -2 - s, -1-t_3(s)) \, : \, -2 < s < 0 \},
\end{align*}
\end{tiny} with $t_2(s)
=\frac{3s+2}{s-2}$ and $t_3(s) = \frac{s^2 + s + 2}{2-s}.$
  \end{example}

\begin{example}\label{example-1e} 
Let $\Lambda$ be as in Examples \ref{example-1a}, 
\ref{example-1b}, \ref{example-1c}, and
 \ref{example-1d}. 
The Jacobi Identity for elements of the
stratum $\calS_\Lambda(\boldR)$ 
reduces to two equations, one for each of the
quadruples $(1,2,3,6) $ and $ (1,2,4,7):$ 
\begin{align*}
 0 & =\alpha_{15}^6 \alpha_{23}^5 - \alpha_{13}^4 \alpha_{24}^6 \\
0&=
 \alpha_{16}^7 \alpha_{24}^6 - \alpha_{14}^5 \alpha_{25}^7 -
\alpha_{12}^3 \alpha_{34}^7  .\end{align*}
There are four sets of  sign choices to consider, one for each element of 
the simple transversal for the $\rho_{\hat Y}$ action.  
We need to solve these equations for elements of some 
set of form $\Sigma(T,\Delta_{\bfa_0})^p.$

For example, if the sign vector is $\bf0,$ 
substituting the values of $\bfa(s,t,u)$ from \eqref{a(s,t,u)}
 into the equations above gives
\begin{align*}
0&=6s - u + su\\
0&=- 2t^2 - 2ut - s + 5u - su.\end{align*}
 This set, and the other components of  $\Sigma(T,\Delta_{\bfa_0})
 \subseteq  \calL_7(\boldR)$
 arising from other sign choices, 
 can be parametrized by established  methods.    
\end{example}

\bibliographystyle{alpha}


\end{document}